\documentclass[11pt]{amsart}

\usepackage[margin=1in]{geometry}
\usepackage{amsmath,amssymb,amsthm,mathtools}
\numberwithin{equation}{section}
\usepackage{booktabs}
\usepackage{microtype}
\usepackage[hidelinks]{hyperref}
\newtheorem{theorem}{Theorem}[section]
\newtheorem{lemma}[theorem]{Lemma}
\newtheorem{proposition}[theorem]{Proposition}
\newtheorem{corollary}[theorem]{Corollary}
\newtheorem{remark}[theorem]{Remark}

\newcommand{\T}{\mathbb T}
\newcommand{\R}{\mathbb R}
\newcommand{\Z}{\mathbb Z}
\newcommand{\Q}{\mathbb Q}
\newcommand{\pr}{\operatorname{pr}}
\newcommand{\rot}{\rho}

\title[Minimal sets and irrational circle factors]
{Minimal sets for torus homeomorphisms with an irrational circle factor}
\author{Xiao-Chuan Liu}
\address[Liu]{Departamento de Matem\'atica,
Universidade Federal de Pernambuco,
Avenida Jornalista An\'ibal Fernandes - Cidade Universit\'aria, Recife, Brazil}
\email{xiaochuan.liu@ufpe.br}
\date{}

\begin{document}

\begin{abstract}
We study minimal sets of torus homeomorphisms admitting an irrational circle factor whose fibres are thin
essential annular continua. For totally irrational pseudo-rotations and homeomorphisms in a nontrivial
Dehn-twist class, we prove uniqueness of the minimal set when the fibres are Jordan curves on a residual set of
base parameters. The same conclusion holds if the fibre cores are Jordan curves, or if the fibres are locally
connected, on a nonmeagre set of parameters. The residual hypothesis cannot be replaced by a full-measure
hypothesis, even under area preservation and topological transitivity. For every totally irrational rotation
vector, we construct such a map with Jordan-curve fibres almost everywhere and uncountably many pairwise disjoint
uniquely ergodic minimal Cantor sets. We also construct examples in every nontrivial Dehn-twist class, with
prescribed irrational vertical rotation number and bounded deviations. In both families, the Jordan-curve
parameters form a meagre set of full Lebesgue measure.
\end{abstract}

\maketitle

\section{Introduction}

An orientation-preserving circle homeomorphism with irrational rotation number has a unique minimal set.
For a torus homeomorphism with an irrational circle factor, we ask how the topology of the fibres affects
uniqueness. We consider totally irrational pseudo-rotations and homeomorphisms in nontrivial Dehn-twist classes,
with every fibre a thin essential annular continuum. In both classes, Jordan-curve fibres on a residual set of
parameters force a unique minimal set. We also construct area-preserving, topologically transitive examples with
Jordan-curve fibres almost everywhere and uncountably many minimal Cantor sets. Together, these results give a
sharp distinction between measure and category.

Write $\T^d=\R^d/\Z^d$, and let $\pr_i$ denote the coordinate projections. A vector $(\alpha,\beta)$ is
\textbf{totally irrational} if $1,\alpha,\beta$ are linearly independent over $\Q$. A \textbf{totally irrational
pseudo-rotation} is a homeomorphism $f\in\operatorname{Homeo}_0(\T^2)$, isotopic to the identity, whose rotation set
is a single totally irrational vector. More precisely, there are a lift $F:\R^2\to\R^2$ of $f$ and a totally
irrational vector $(\alpha,\beta)$ such that
\begin{equation}
 \lim_{n\to\infty}\sup_{z\in\R^2}
 \left\|\frac{F^n(z)-z}{n}-(\alpha,\beta)\right\|=0.
\end{equation}

Rotation vectors describe average displacements of lifted orbits. We give the precise definitions in
Section~\ref{sec:preliminaries}. Potrie proved that every nonwandering map in this class is topologically
transitive \cite[Corollary~1.1]{Potrie2012}.
J{\"a}ger, Kwakkel and Passeggi classified the possible proper minimal sets of nonwandering torus homeomorphisms.
In the totally irrational nonwandering case, every proper minimal set is an extension of a Cantor set
\cite[Corollary~5(a)]{JagerKwakkelPasseggi2013}: a torus semiconjugacy homotopic to the identity maps it onto a
Cantor minimal set. The question whether the minimal set must be unique was raised by
Kwakkel \cite[Question~1]{Kwakkel2011} and discussed by Potrie \cite[Introduction, p.~3974]{Potrie2012}.
We study coexistence of minimal sets and conditions on the fibres which rule it out.

We consider maps admitting an \textbf{irrational circle factor}: a continuous surjection $h:\T^2\to\T^1$ satisfying,
for some $\alpha\in\R\setminus\Q$,
\begin{equation}\label{eq:semiconjugacy}
 h\circ f=R_\alpha\circ h,\qquad R_\alpha(\theta)=\theta+\alpha.
\end{equation}
In the identity class we take $h$ homotopic to $\pr_1$. A fibre is \textbf{thin} if it has empty interior. An
\textbf{essential annular continuum} is a compact connected subset of the torus whose complement is an open annulus.

An unpublished example of Avila \cite{AvilaPrivateCommunication}, which has circulated among specialists, gives a
negative answer to the general uniqueness question. For any prescribed totally irrational rotation vector, it has uncountably many pairwise disjoint uniquely
ergodic minimal Cantor sets. Over two rotation orbits, the fibres of its circle factor are closed annuli with
nonempty interior; every other fibre is a Jordan curve. We give the construction and its proof in
Proposition~\ref{prop:denjoy-model}.

The interiors of these annuli are wandering, so Avila's model does not preserve area. Conservative examples with
singular thin fibres were constructed by B{\'e}guin, Crovisier and J{\"a}ger \cite{BeguinCrovisierJager2017}:
their real-analytic area-preserving pseudo-rotation is minimal, and every fibre of its circle factor is a
pseudo-circle. Our first theorem gives uncountably many minimal Cantor sets in an area-preserving map while
retaining Jordan fibres almost everywhere. The rotation vector can be any prescribed totally irrational vector.

\begin{theorem}\label{thm:thin-counterexample}
Let $(\alpha,\beta)\in\R^2$ be totally irrational. There exist an area-preserving, topologically transitive
homeomorphism $f\in\operatorname{Homeo}_0(\T^2)$, a lift $F$, and a continuous surjection
$h:\T^2\to\T^1$ homotopic to $\pr_1$ such that
\begin{equation}\label{eq:main-counterexample}
 \rot(F)=\{(\alpha,\beta)\},\qquad h\circ f=R_\alpha\circ h.
\end{equation}
Every fibre of $h$ is a thin essential annular continuum. The parameters whose fibres are Jordan curves have full
Lebesgue measure. Moreover, $f$ has uncountably many pairwise disjoint uniquely ergodic minimal Cantor sets.
\end{theorem}

The same coexistence occurs in every nontrivial Dehn-twist class. Write $\operatorname{Homeo}_k(\T^2)$ for the
homeomorphisms isotopic to $(x,y)\mapsto(x+ky,y)$. Addas-Zanata, Tal and Garcia proved that a minimal homeomorphism
in such a class has a singleton irrational vertical rotation set \cite[Theorem~1]{AddasZanataTalGarcia2014}.
Our next theorem prescribes the irrational vertical rotation number and the nonzero twisting degree while
giving an area-preserving, transitive map with uncountably many minimal Cantor sets. The factor is vertical.
A lift has $\alpha$-bounded vertical deviations if its vertical displacement after $n$ iterates differs from
$n\alpha$ by a bound independent of the point and of $n\in\Z$. See Section~\ref{sec:preliminaries}.

\begin{theorem}\label{thm:dehn-counterexample}
Let $k\in\Z\setminus\{0\}$ and $\alpha\in\R\setminus\Q$. There exist an area-preserving, topologically transitive
homeomorphism $f\in\operatorname{Homeo}_k(\T^2)$, a lift $F$, and a continuous surjection $h:\T^2\to\T^1$ homotopic to
$\pr_2$ such that
\begin{equation}\label{eq:dehn-factor}
 h\circ f=R_\alpha\circ h.
\end{equation}
The lift $F$ has $\alpha$-bounded vertical deviations. Every fibre of $h$ is a thin essential annular continuum.
The parameters whose fibres are Jordan curves have full Lebesgue measure. Moreover, $f$ has uncountably many pairwise
disjoint uniquely ergodic minimal Cantor sets.
\end{theorem}

For circle skew products in the identity class over an irrational rotation, B{\'e}guin, Crovisier, J{\"a}ger and
Le Roux constructed transitive nonminimal examples with a unique minimal Cantor set
\cite[Proposition~1.4]{BeguinCrovisierJagerLeRoux2009}. They also proved uniqueness under transitivity or the
absence of invariant strips \cite[Proposition~4.2]{BeguinCrovisierJagerLeRoux2009}. For circle skew products with
nonzero twisting degree, Hammerlindl and Potrie proved uniqueness
\cite[Section~6.2, proof of Proposition~6.3]{HammerlindlPotrie2014}. The following theorem extends uniqueness
to factors with singular annular fibres. It suffices that the fibres are Jordan curves on a residual set of
parameters. A set is residual if it contains a dense $G_\delta$ subset.

\begin{theorem}\label{thm:generic-jordan-fibres}
Let $h:\T^2\to\T^1$ be a continuous surjection. Suppose that either
\begin{enumerate}
 \item $f$ is a totally irrational pseudo-rotation and $h\simeq\pr_1$ satisfies \eqref{eq:semiconjugacy}; or
 \item $f\in\operatorname{Homeo}_k(\T^2)$ for some $k\neq0$ and $h\simeq\pr_2$ satisfies \eqref{eq:dehn-factor}, with
 $\alpha$ irrational.
\end{enumerate}
Assume that every fibre of $h$ is a thin essential annular continuum. If the fibres are Jordan curves on a
residual set of base parameters, then $f$ has a unique minimal set.
\end{theorem}

Every thin annular continuum has a unique \textbf{core}, its inclusion-minimal essential annular subcontinuum.
Only countably many fibres differ from their cores (Proposition~\ref{prop:countable-spikes}). This gives another
criterion for uniqueness. Here nonmeagre means not a countable union of nowhere dense sets.

\begin{corollary}\label{cor:locally-connected}
Under either dynamical hypothesis of Theorem~\ref{thm:generic-jordan-fibres}, suppose that every fibre is a thin
essential annular continuum. If the fibre cores are Jordan curves on a nonmeagre set of base parameters, then
$f$ has a unique minimal set. The same conclusion holds if the fibres are locally connected on a nonmeagre set.
\end{corollary}

The unique minimal set may be a proper subset of the torus. Neither nonwandering dynamics nor a common
modulus of local connectedness is required. Theorems~\ref{thm:thin-counterexample} and~\ref{thm:dehn-counterexample}
show that the residual hypothesis in Theorem~\ref{thm:generic-jordan-fibres} cannot be replaced by a full-measure
hypothesis, even after adding area preservation and transitivity. In both examples, the parameters whose
fibres are Jordan curves form a meagre set of full Lebesgue measure. Equivalently, the parameters whose
fibres are not Jordan curves form a residual set of measure zero.

The constructions use approximation by conjugation of a fixed Denjoy model. The main issue is to obtain an
invariant measure of full support while preserving a prescribed family of minimal sets and Jordan fibres almost
everywhere. Each conjugacy fixes every point in the prescribed minimal sets. Successive perturbations give
positive mass to the sets of a countable basis and preserve the earlier positive-mass conditions. We also require
summable bounds on the measure of the base parameters affected by the perturbations. The limiting measure has
full support and projects to Lebesgue measure, so every fibre has measure zero and empty interior. For almost
every parameter, the image of the source circle stabilises to an essential Jordan curve in the limiting fibre.
Proposition~\ref{prop:countable-spikes} identifies this curve with the whole fibre outside a countable exceptional
set. The Oxtoby--Ulam theorem gives area preservation.

For uniqueness, the main issue is to extend cyclic-order arguments from circle skew products to singular annular
fibres. Lemma~\ref{lem:prime-end-good-fibre} gives the required prime-end control. At a residual set of parameters,
the fibre is a Jordan curve and every impression of every nearby fibre stays close to the corresponding point of
that curve. The estimate is uniform over all prime ends on both sides; the nearby fibres may be singular. For any
fixed compact set meeting every fibre, the lemma also gives continuity of the intersections with the fibres and of
the corresponding prime-end sets. It uses no dynamics and no common modulus of local connectedness. These estimates
show that compact subarcs of gaps between disjoint compact sets persist on nearby prime-end circles. This extends
the oriented-gap argument to the annular decompositions considered here.

After the preliminaries, Section~\ref{sec:counterexample} gives Avila's example and proves
Theorem~\ref{thm:thin-counterexample}. Section~\ref{sec:dehn} treats the Dehn-twist class.
Section~\ref{sec:thin-geometry} establishes the fibre structure and reduces Corollary~\ref{cor:locally-connected}
to Theorem~\ref{thm:generic-jordan-fibres}. Section~\ref{sec:uniqueness} proves that theorem.

\section{Notation and preliminaries}\label{sec:preliminaries}

Let $\pi:\R^2\to\T^2$ be the covering map. A lift $F:\R^2\to\R^2$ of a torus homeomorphism $f$ satisfies
$\pi F=f\pi$. If $f\in\operatorname{Homeo}_0(\T^2)$, then $F$ commutes with integer translations.
The displacement $F-\operatorname{id}$ is bounded and $\Z^2$-periodic. The \textbf{rotation set} of $F$ is
\begin{equation}\label{eq:rotation-set}
 \rot(F)=\bigcap_{N\geq1}
 \overline{\left\{\frac{F^n(z)-z}{n}:z\in\R^2,\ n\geq N\right\}}.
\end{equation}
Thus $\rot(F)$ consists of all limits of $(F^{n_j}(z_j)-z_j)/n_j$ with $n_j\to\infty$. The initial points $z_j$ may
vary. Changing $F$ by an integer translation translates $\rot(F)$ by the same vector. Thus being a totally
irrational pseudo-rotation does not depend on the choice of lift.

If $\rot(F)=\{\rho\}$, boundedness of the displacement and the definition of the rotation set give
\begin{equation}\label{eq:uniform-rotation}
 \frac{F^n(z)-z}{n}\longrightarrow\rho
 \quad\text{uniformly in }z\in\R^2.
\end{equation}
Indeed, failure of uniform convergence would give a sequence of average displacements bounded away from $\rho$.
A convergent subsequence would then yield another point of $\rot(F)$.

If $f\in\operatorname{Homeo}_k(\T^2)$, its action on first homology is given by
\begin{equation}\label{eq:dehn-matrix}
 A_k=\begin{pmatrix}1&k\\0&1\end{pmatrix}.
\end{equation}
A lift satisfies
$F(z+v)=F(z)+A_kv$ for $v\in\Z^2$. Its vertical displacement is therefore periodic, even when its full displacement
is not. The \textbf{vertical rotation set} is
\begin{equation}\label{eq:vertical-rotation-set}
 \rho_V(F)=\bigcap_{N\geq1}\overline{
 \left\{\frac{\pr_2(F^n(z)-z)}{n}:z\in\R^2,\ n\geq N\right\}}.
\end{equation}
The lift $F$ has \textbf{$\alpha$-bounded vertical deviations} if
\begin{equation}\label{eq:vertical-bounded-deviations}
 \sup_{z\in\R^2,\,n\in\Z}
 \left|\pr_2(F^n(z)-z)-n\alpha\right|<\infty.
\end{equation}
This implies $\rho_V(F)=\{\alpha\}$. A vertical circle factor gives such a bound by
Lemma~\ref{lem:factor-bounded-deviation}.

Area preservation means preservation of normalised Lebesgue measure on $\T^2$.
A minimal set is a nonempty compact invariant set with no proper nonempty compact invariant subset. It is uniquely
ergodic if the restricted dynamics carries exactly one invariant Borel probability measure. A homeomorphism is nonwandering
if every nonempty open set $U$ meets $f^n(U)$ for some $n\geq1$. We write $\Omega(f)$ for its nonwandering set.
It is topologically transitive if, for any nonempty open sets $U,V$, some $n\geq1$ satisfies $f^n(U)\cap V\neq\varnothing$.

We use the maximum product metric on $\T^2$ and the uniform metric $d_0$ on spaces of continuous maps. On torus
homeomorphisms put
\begin{equation}\label{eq:homeomorphism-metric}
 d_{\mathrm H}(f,g)=d_0(f,g)+d_0(f^{-1},g^{-1}).
\end{equation}
This metric is complete, and each homotopy class is closed.
Homotopy classes in $C(\T^2,\T^1)$ are also closed in the uniform metric.
We use the Hausdorff metric for nonempty compact sets.
Statements about the measure or category of fibre parameters refer to the base circle.

If the factor map is a circle bundle, bundle coordinates make the dynamics a circle skew product. In the totally
irrational identity case, an invariant strip would give a rational relation among
$1,\alpha,\beta$ \cite[Remark~3.8 and Lemma~3.9]{JagerKeller2006}. Its absence gives a unique minimal set
\cite[Proposition~4.2]{BeguinCrovisierJagerLeRoux2009}.

\subsection{Results used in the proofs}

\begin{lemma}\label{lem:measure-facts}
Let $Y$ and $Z$ be compact metric spaces.
\begin{enumerate}
\item\label{item:measure-limits} Suppose that continuous maps $T_n:Y\to Y$ and $p_n:Y\to Z$ converge uniformly to $T$ and $p$,
and that probability measures $\nu_n$ converge weakly to $\nu$. Then
$(T_n)_*\nu_n\to T_*\nu$ and $(p_n)_*\nu_n\to p_*\nu$ weakly. In particular, if
$(T_n)_*\nu_n=\nu_n$ and $(p_n)_*\nu_n=\lambda$, then $T_*\nu=\nu$ and $p_*\nu=\lambda$.
\item\label{item:uniform-ergodic} If $T:Y\to Y$ is continuous and has a unique invariant probability measure
$\nu$, then for every $v\in C(Y,\R)$,
\begin{equation}\label{eq:prelim-uniform-ergodic}
 \frac1n\sum_{j=0}^{n-1}v(T^j y)\longrightarrow\int_Y v\,d\nu
 \quad\text{uniformly in }y\in Y.
\end{equation}
\item\label{item:measure-recurrence} If a homeomorphism $T:Y\to Y$ preserves a probability measure $\nu$,
every measurable set $W$ of positive measure satisfies $T^n(W)\cap W\neq\varnothing$ for arbitrarily large
positive integers $n$. In particular, a full-support invariant probability measure makes $T$ nonwandering.
\item\label{item:summable-sets} If $(E_n)$ is a sequence of measurable sets in a probability space $(Y,\nu)$
and $\sum_n\nu(E_n)<\infty$, then almost every point belongs to only finitely many $E_n$.
\end{enumerate}
\end{lemma}

\begin{proof}
For the first assertion, test against a continuous function. Uniform convergence controls the change in the map,
and weak convergence controls the change in the measure. For the second, every weak limit of empirical measures
$n^{-1}\sum_{j=0}^{n-1}\delta_{T^j y_n}$ is invariant, even when the initial points $y_n$ vary. Uniqueness of
$\nu$ therefore gives the asserted uniform convergence.

For the third, suppose that all positive return times are at most $N$.
Then the sets $T^{j(N+1)}(W)$, $j\geq0$, are pairwise disjoint and have the same positive measure.
This contradicts finiteness of the measure. The last
assertion follows from
$\nu(\bigcup_{n\geq N}E_n)\leq\sum_{n\geq N}\nu(E_n)\to0$.
\end{proof}

The following is the Oxtoby--Ulam theorem in the isotopy form recalled in
\cite[Section~2, pp.~150--151]{LeRoux2014}.

\begin{lemma}\label{lem:oxtoby-ulam}
Let $\nu$ be a nonatomic Borel probability measure of full support on $\T^2$. There is a homeomorphism
$\Phi\in\operatorname{Homeo}_0(\T^2)$ such that $\Phi_*\nu=\operatorname{Leb}_{\T^2}$. Consequently, if
$f_*\nu=\nu$, then $\Phi f\Phi^{-1}$ preserves normalised Lebesgue measure.
\end{lemma}

For rotation sets and recurrence, we use the following two results. The first is
\cite[Theorem~1]{KoropeckiPasseggiSambarino2021}. It requires no recurrence assumption.

\begin{lemma}\label{lem:rotation-singleton}
Let $f\in\operatorname{Homeo}_0(\T^2)$. If a continuous surjection $h:\T^2\to\T^1$ satisfies
$hf=R_\alpha h$ with $\alpha$ irrational, then $\rot(F)$ is a singleton for every lift $F$ of $f$.
\end{lemma}

The next statement combines \cite[Theorem~A and Corollary~1.1]{Potrie2012}.

\begin{lemma}\label{lem:potrie-transitivity}
Let $f$ be a totally irrational pseudo-rotation. If open sets $U,V\subset\T^2$ both meet $\Omega(f)$, then
$f^n(U)\cap V\neq\varnothing$ for some $n\geq1$. In particular, if $f$ is nonwandering, then it is topologically
transitive.
\end{lemma}

Recall that an essential annular continuum in the torus has an open annulus as its complement. In the
open annulus $\mathbb A=\R\times\T^1$, an essential annular continuum is a compact continuum whose complement consists of exactly two
components, one containing each end. A \textbf{circloid} is an inclusion-minimal essential annular continuum.
We use the following consequence of \cite[Theorem~1 and Corollary~4.3]{JagerPasseggi2015}.

\begin{lemma}\label{prop:canonical-fibres}
Let $f$ be a totally irrational pseudo-rotation and let $h:\T^2\to\T^1$ be a continuous surjection homotopic to a
coordinate projection. If $hf=R_\alpha h$, where $\alpha$ is irrational, then every fibre of $h$ is an essential
annular continuum.
\end{lemma}

\begin{proof}
The map $h$ is homotopic to a coordinate projection. Thus the construction in
\cite[Section~3, Lemmas~3.1--3.3]{JagerPasseggi2015} gives a semiconjugacy $h_0$ to $R_\alpha$
with essential annular fibres. Lemma~\ref{lem:potrie-transitivity} verifies the external-transitivity
hypothesis on the nonwandering set in \cite[Corollary~4.3]{JagerPasseggi2015}. That corollary gives
$h=R_c h_0$ for some $c\in\T^1$.
The assertion follows.
\end{proof}

For the Dehn-twist case, we use the terminology of \cite[Section~2.4.2]{Kocsard2021}. An open subset of the torus
is \textbf{inessential} if every loop in it is contractible in the torus. A subset is inessential if it has an
inessential open neighbourhood. A set is \textbf{fully essential} if its complement is inessential.
For a domain $U\subset\T^2$, this means that $\T^2\setminus U$ is contained in a topological disc.
A torus homeomorphism is \textbf{eventually annular} if some iterate $f^q$, $q\geq1$, has a lift $G$ and a nonzero vector $v\in\Z^2$ such
that $|\langle G^n(z)-z,v\rangle|$ is bounded uniformly in $z\in\R^2$ and $n\in\Z$.
The following statement is \cite[Proposition~2.11]{Kocsard2021}.

\begin{lemma}\label{lem:fully-essential-recurrence}
Let $f$ be a torus homeomorphism without periodic points. Assume that $f$ is not eventually annular, and let
$x\in\Omega(f)$. For every sufficiently small open disc $B$ containing $x$, the component of
$\bigcup_{n\in\Z}f^n(B)$ containing $B$ is fully essential.
\end{lemma}

We next recall the topology of thin annular continua. Let $A\subset\mathbb A$ be an essential annular continuum.
Write $U^-(A)$ and $U^+(A)$ for the components of $\mathbb A\setminus A$ containing the lower and upper ends,
respectively. Their boundaries are denoted by $A^-=\partial U^-(A)$ and $A^+=\partial U^+(A)$. The following facts are given in
\cite[Lemma~2.3, the discussion before Lemma~5.7, and Lemma~5.7]{JagerPasseggi2015}.

\begin{lemma}\label{lem:thin-annular-sides}
If $A\subset\mathbb A$ is a thin essential annular continuum, then $A=A^-\cup A^+$ and $A^-\cap A^+$ is the
unique circloid contained in $A$.

Let $p:\mathbb A\to\R$ be proper and continuous. Assume that $p$ tends to $-\infty$ at the lower end and to
$+\infty$ at the upper end. Suppose that every fibre $A_t=p^{-1}(t)$ is a thin essential annular continuum. Then
\begin{equation}\label{eq:prelim-sided-limits}
 \lim_{s\nearrow t}^{\mathcal H}A_s=A_t^- ,
 \qquad
 \lim_{s\searrow t}^{\mathcal H}A_s=A_t^+ .
\end{equation}
\end{lemma}

The following elementary facts connect local connectedness with Jordan curves.

\begin{lemma}\label{lem:locally-connected-separator}
A continuous image of a circle in a metric space is locally connected. Every locally connected continuum which
separates the sphere contains a Jordan curve.
\end{lemma}

\begin{proof}
Uniform continuity gives a finite cover of the image of a circle by connected compact sets of arbitrarily small
diameter. Fix a point in the image and take the union of the sets containing it. This union is connected.
The remaining sets are compact and do not contain the point, so the union contains a relative neighbourhood.
We obtain connected neighbourhoods of arbitrarily small diameter. Thus the image is locally connected.

A locally connected continuum containing no Jordan curve is a dendrite
\cite[Section~6.1]{HrushovskiLoeserPoonen2014}. It is contractible by
\cite[Proposition~6.5]{HrushovskiLoeserPoonen2014}, so Alexander duality implies that its complement in the sphere
is connected. This proves the second assertion.
\end{proof}

For the topological theory of prime ends on surfaces, see \cite{Mather1982} and
\cite[Section~3]{KoropeckiLeCalvezNassiri2015}; for the conformal viewpoint, see
\cite[Chapter~2]{Pommerenke1992}.

The conformal facts below are used only after capping an end of the annulus. Write
$\mathbb D=\{z\in\mathbb C:|z|<1\}$. We use the prime-end description in \cite[Section~3]{Rempe2008} and the
Carath\'eodory boundary-extension theorem in \cite[Section~2.3]{Pommerenke1992}.

\begin{lemma}\label{lem:prime-end-facts}
Let $D\subset\mathbb C$ be a bounded simply connected domain and let $\varphi:\mathbb D\to D$ be a Riemann map.
It identifies $\overline{\mathbb D}$ with the prime-end compactification of $D$. The impression associated to
$\zeta\in\partial\mathbb D$ is the full cluster set
\begin{equation}\label{eq:prelim-impression}
 I(\zeta)=\{\lim_{n\to\infty}\varphi(z_n):z_n\in\mathbb D,\ z_n\to\zeta\}.
\end{equation}
The impressions cover $\partial D$. The set of pairs
$\{(\zeta,x):\zeta\in\partial\mathbb D,\ x\in I(\zeta)\}$ is a compact subset of
$\partial\mathbb D\times\partial D$. In particular, if $\zeta_n\to\zeta$ and $x_n\to x$ with
$x_n\in I(\zeta_n)$, then $x\in I(\zeta)$.
If $\partial D$ is a Jordan curve, then $\varphi$ extends to a homeomorphism
$\overline{\mathbb D}\to\overline D$ and $I(\zeta)=\{\varphi(\zeta)\}$.
\end{lemma}

\begin{proof}
The compactification and the description of impressions are recalled in \cite[Section~3]{Rempe2008}.
The Jordan case follows from the boundary-extension theorem cited above. Since
$\partial\mathbb D\times\partial D$ is compact, it suffices to show that the set of pairs is closed.
Let $\zeta_n\to\zeta$ and $x_n\to x$ with $x_n\in I(\zeta_n)$. Choose $z_n\in\mathbb D$ such that
$|z_n-\zeta_n|<1/n$ and $|\varphi(z_n)-x_n|<1/n$. Then $z_n\to\zeta$ and $\varphi(z_n)\to x$, so
$x\in I(\zeta)$. Finally, approach any point of
$\partial D$ from inside $D$ and take a convergent subsequence of the inverse images. Its limit lies on
$\partial\mathbb D$ and gives an impression containing the chosen point.
\end{proof}

We use the following two forms of the Carath\'eodory kernel theorem
\cite[Theorem~1.8]{Pommerenke1992}.

\begin{lemma}\label{lem:riemann-map-convergence}
Let $D_n,D\subset\mathbb C$ be bounded simply connected domains containing $0$. Assume that they contain a common
disc centred at $0$ and lie in a common bounded disc. Let the Riemann maps $\varphi_n:\mathbb D\to D_n$ and
$\varphi:\mathbb D\to D$ be normalised by $\varphi_n(0)=\varphi(0)=0$, with their complex derivatives
$\varphi_n'(0)$ and $\varphi'(0)$ positive real numbers.
Assume also that one of the following conditions holds:
\begin{enumerate}
\item\label{item:increasing-domains} $D_n\subset D_{n+1}$ and $D=\bigcup_nD_n$;
\item\label{item:jordan-domain-limit} $\partial D$ is a Jordan curve and $\partial D_n\to\partial D$ in the
Hausdorff metric.
\end{enumerate}
Then $\varphi_n\to\varphi$ locally uniformly on $\mathbb D$.
\end{lemma}

\begin{proof}
In the first case, the kernel of every subsequence is $D$. The same holds in the second case.
To see this, fix a compact subset of $D$. It lies in a compact connected subset of $D$ containing $0$.
Hausdorff convergence of the boundaries puts this larger set inside $D_n$ for all large $n$.
Similarly, each point outside $\overline D$ can be joined to infinity by a path avoiding $\partial D$.
The boundaries $\partial D_n$ also avoid this path for all large $n$, so the point lies outside $D_n$. The Jordan boundary separates these two regions. The kernel theorem
now gives the claimed convergence in both cases.
\end{proof}

We will need parameters whose fibres are Jordan curves and at which two kinds of maps are continuous:
the boundary maps and the maps giving intersections with fixed compact sets. The following category facts provide them.
The first assertion is the Baire continuity theorem \cite[Theorem~24.14]{Kechris1995}.
The second follows from the first.

\begin{lemma}\label{lem:baire-continuity}
Let $T$ be a Polish space.
\begin{enumerate}
\item\label{item:baire-functions} If $Y$ is a separable metric space and $F:T\to Y$ is a pointwise limit of
continuous maps, then $F$ has a dense $G_\delta$ set of continuity points.
\item\label{item:baire-sections} Let $Z$ be a compact metric space. Write $\mathcal K(Z)$ for the space of nonempty
compact subsets of $Z$, with the Hausdorff metric. If $A:T\to\mathcal K(Z)$ is upper semicontinuous,
then $A$ has a dense $G_\delta$ set of continuity points.
\end{enumerate}
\end{lemma}

\begin{proof}
The first assertion is the cited theorem. For the second, fix a countable dense subset $Z_0\subset Z$.
For each $z\in Z_0$, the function $t\mapsto d(z,A(t))$ is lower semicontinuous and hence Baire-one.
Intersect their dense $G_\delta$ sets of continuity points. At a parameter in this intersection, the distance
functions converge at every point of $Z_0$. They are all $1$-Lipschitz in $z$. Approximation by finite nets
therefore gives uniform convergence on $Z$. This is equivalent to Hausdorff convergence of the compact sets.
\end{proof}

\begin{lemma}\label{lem:category-zero-one}
Let $\alpha\in\R\setminus\Q$. An $R_\alpha$-invariant subset of $\T^1$ with the Baire property is either meagre or residual.
\end{lemma}

\begin{proof}
Let $D$ be such a set and suppose that it is nonmeagre. The Baire property gives a nonempty open interval $I$ on which
$D$ is residual. Invariance shows that $D$ is residual in every $R_{n\alpha}(I)$, $n\in\Z$. These intervals cover the
circle by minimality of $R_\alpha$. Hence
\begin{equation}
 \T^1\setminus D
 \subset \bigcup_{n\in\Z}\bigl(R_{n\alpha}(I)\setminus D\bigr)
\end{equation}
is meagre, and $D$ is residual.
\end{proof}

\section{Many minimal sets}\label{sec:counterexample}

We first give Avila's example. We then prove Theorem~\ref{thm:thin-counterexample} by taking limits of its
conjugates. The conjugating maps fix every point in a Cantor family of minimal sets. We control the dynamics,
the circle factors and the invariant measures throughout the construction.

The relative perturbations and the use of open dense conditions are related to \cite[Section~4]{Potrie2012}.
Potrie preserves a fixed minimal system and enlarges the nonwandering set. His example has a unique minimal set.
Here we retain an entire Cantor family of minimal systems.

\subsection{Avila's example}\label{sec:model}

We use the following formulation of Avila's unpublished construction \cite{AvilaPrivateCommunication}.
The construction blows up two distinct rotation orbits. The first blow-up permits a change in the
second-coordinate displacement across a wandering interval. The second permits the opposite change. These two changes give a continuous
function on the circle with the prescribed average.

\begin{proposition}[Avila \cite{AvilaPrivateCommunication}]\label{prop:denjoy-model}
Let $\rho=(\alpha,\beta)$ be totally irrational. There exist an orientation-preserving Denjoy homeomorphism
$f_0:\T^1\to\T^1$ with unique minimal Cantor set $K$, a monotone degree-one map $\psi:\T^1\to\T^1$, and a continuous
function $u:\T^1\to\R$ with the following properties. The maps
\begin{equation}\label{eq:source-map}
 g(x,y)=(f_0(x),y+u(x)\bmod1),\qquad h_0(x,y)=\psi(x)
\end{equation}
satisfy $h_0g=R_\alpha h_0$, and $g$ has a lift $G$ with $\rot(G)=\{\rho\}$. The closures of the wandering intervals
of $f_0$ form the two distinct orbits of two closed intervals $I^0,I^1$. The function $u$ is integer-valued and locally
constant on $\T^1\setminus(\operatorname{int}I^0\cup\operatorname{int}I^1)$. The minimal sets of $g$ are exactly
$K\times\{s\}$, $s\in\T^1$, and each is uniquely ergodic.
\end{proposition}

\begin{proof}
Write $\beta=k+b$, where $k\in\Z$ and $0<b<1$. Total irrationality implies that $0$ and $b$ lie on distinct
$R_\alpha$-orbits. Perform a Denjoy blow-up of these two orbits.  We obtain an orientation-preserving circle
homeomorphism $f_0$ and a monotone degree-one map $\psi:\T^1\to\T^1$ such that
\begin{equation}\label{eq:denjoy-factor}
 \psi\circ f_0=R_\alpha\circ\psi.
\end{equation}
The map $f_0$ has a unique minimal Cantor set $K$. It is also uniquely ergodic. Indeed, $\psi$ sends every $f_0$-invariant probability
measure to Lebesgue measure. Each of the countably many nontrivial fibres has measure zero. Outside their union,
$\psi$ is injective.  Denote the unique invariant measure by $\lambda$.  It is supported on $K$ and
satisfies
\begin{equation}\label{eq:push-measure}
 \psi_*\lambda=\operatorname{Leb}_{\T^1}.
\end{equation}

We construct a continuous function $v:\T^1\to[0,1]$. Give $\T^1$ its positive orientation. Set
$v=1$ on $K\cap\psi^{-1}((0,b))$ and $v=0$ on $K\cap\psi^{-1}((b,1))$. Write $\psi^{-1}(0)=[a_0,b_0]$ and
$\psi^{-1}(b)=[a_b,b_b]$ in the positive orientation.  Set $v(a_0)=0$, $v(b_0)=1$, $v(a_b)=1$ and $v(b_b)=0$.
Interpolate continuously across these two intervals. Extend $v$ constantly across every other complementary interval
of $K$.  The restriction to $K$ is continuous, since the possible changes occur across the two chosen intervals.  This
defines a continuous function with
\begin{equation}\label{eq:v-properties}
 v(K)\subset\{0,1\},
 \qquad
 \int_{\T^1}v\,d\lambda=b.
\end{equation}
The integral follows from \eqref{eq:push-measure}, since the two boundary fibres have zero $\lambda$-measure. Put
$u=k+v$. Its restriction to the complement of the interiors of $I^0=\psi^{-1}(0)$ and $I^1=\psi^{-1}(b)$ is
integer-valued and locally constant.

Choose a lift $F_0:\R\to\R$ of $f_0$ with rotation number $\alpha$, and let $\widetilde u:\R\to\R$ be the $1$-periodic
function induced by $u$.  Define
\begin{equation}\label{eq:denjoy-skew-product}
 G(x,y)=\bigl(F_0(x),y+\widetilde u(x)\bigr).
\end{equation}
This is a lift of the torus homeomorphism $g(x,y)=\bigl(f_0(x),y+u(x)\pmod 1\bigr)$.

For each $s\in\T^1$, the set $K\times\{s\}$ is $g$-invariant because $u$ is integer-valued on $K$.
The restriction of $g$ to this set is conjugate to $f_0|_K$. It is therefore uniquely ergodic and minimal.
These Cantor sets are pairwise disjoint. Conversely, the first-coordinate projection of any minimal set
for $g$ is $K$. On $K\times\T^1$, the second coordinate is invariant. It is therefore constant on every minimal
set. Hence the minimal sets of $g$ are exactly $K\times\{s\}$, $s\in\T^1$.

Lemma~\ref{lem:measure-facts}\eqref{item:uniform-ergodic} and \eqref{eq:v-properties} give
\begin{equation}\label{eq:birkhoff-u}
 \frac1n\sum_{j=0}^{n-1}u(f_0^j(x))
 \longrightarrow \int_{\T^1}u\,d\lambda=k+b=\beta
 \quad\text{uniformly in }x.
\end{equation}
The horizontal displacement of $F_0^n$ divided by $n$ converges uniformly to $\alpha$.  Applying these facts to
\eqref{eq:denjoy-skew-product} shows
\begin{equation}
 \frac{G^n(x,y)-(x,y)}{n}\longrightarrow(\alpha,\beta)
 \quad\text{uniformly in }(x,y),
\end{equation}
and therefore $\rot(G)=\{(\alpha,\beta)\}$.

The map $h_0(x,y)=\psi(x)$ semiconjugates $g$ to $R_\alpha$ and is homotopic to the first-coordinate projection. If
$t$ belongs to either blown-up orbit, then $h_0^{-1}(t)=\psi^{-1}(t)\times\T^1$ is a closed annulus with nonempty
interior. Every other fibre is a circle.
\end{proof}

\subsection{The thin-fibre construction and proof of Theorem~\ref{thm:thin-counterexample}}\label{sec:towers}

We first outline the proof. We use \textbf{approximation by conjugation}, in the spirit of the Anosov--Katok
method \cite[Section~2.1]{FayadKatok2004}. The initial torus map in Proposition~\ref{prop:denjoy-model}
remains fixed. Its conjugates, their inverses and their circle factors converge uniformly. The corresponding
invariant measures converge weakly. The conjugacies fix pointwise the union of a Cantor family of minimal sets.
At each stage, we give positive mass to the next member of a countable basis of open balls. We choose the later
perturbations small enough to preserve all earlier positivity conditions. The limiting measure therefore has full support.

The limiting factor sends this measure to Lebesgue measure on the circle. Each fibre therefore has measure zero. Full support then forces it to have empty interior. At every finite stage, some fibres are closed annuli with nonempty interior. All fibres become thin in the limit.
Summable bounds on the affected base parameters ensure that almost every limiting fibre is a Jordan curve.
Lemma~\ref{lem:oxtoby-ulam} then gives a final change of coordinates which makes the limiting map area-preserving.

Recall the construction in Proposition~\ref{prop:denjoy-model}. The Denjoy homeomorphism $f_0$ has a unique minimal
Cantor set $K$. Its unique invariant probability measure $\lambda$ is supported on $K$ and satisfies \eqref{eq:push-measure}.
The initial torus homeomorphism is $g(x,y)=(f_0(x),y+u(x)\bmod1)$. Its circle factor is $h_0(x,y)=\psi(x)$, so
$h_0g=R_\alpha h_0$. Since $u$ is integer-valued on $K$, we have $g(x,s)=(f_0(x),s)$ for $x\in K$.

Fix any Cantor set $S\subset\T^1$ for the rest of this construction, and put
\begin{equation}
 X=K\times S,\qquad \mu_0=\lambda\times\operatorname{Leb}_{\T^1}.
\end{equation}
Thus $X$ is the union of the pairwise disjoint minimal sets $K\times\{s\}$, $s\in S$. The conjugating maps used below
fix $X$ pointwise. They therefore preserve the dynamics on each of these sets. The measure $\mu_0$ is $g$-invariant and has
support $K\times\T^1$.

Fix one of the intervals $I^0,I^1$ and write it as $I=[a,b]$. Put $A=I\times\T^1$.
Choose a real circle coordinate around $I$. The next lemma moves a boundary point of $A$ into a prescribed open
subset of its interior. It controls the conjugated dynamics and the circle factor. It also gives two support estimates
for the limiting construction.

\begin{lemma}\label{lem:relative-tower}
For every $\varepsilon,\eta,\delta>0$ and every nonempty open set $B\subset\operatorname{int}A$, there are a
homeomorphism $P$, isotopic to the identity relative to $X$, and a point $z\in\partial A$ such that
\begin{align}
 d_0(PgP^{-1},g)&<\varepsilon,
 &d_0(Pg^{-1}P^{-1},g^{-1})&<\varepsilon,
 \label{eq:relative-local-dynamics}\\
 d_0(h_0P^{-1},h_0)&<\eta,
 &P(z)&\in B,
 \label{eq:relative-local-factor}\\
 \mu_0(\operatorname{supp}P)&<\delta,
 &\operatorname{Leb}_{\T^1}\bigl(h_0(\operatorname{supp}P)\bigr)&<\delta.
 \label{eq:relative-local-support}
\end{align}
\end{lemma}

\begin{proof}
Choose a fixed interval $J_*\supset I$ whose closure avoids the other interval among $I^0,I^1$. Take
$(x_0,y_0)\in B$ with $y_0\notin S$. Choose an interval $V\ni y_0$ whose closure misses $S$.
An increasing piecewise-linear homeomorphism of $J_*$, equal to the identity near its endpoints, can send $a$ to
$x_0$. Interpolate it linearly with the identity. Multiply the interpolation parameter by a cutoff supported in
$V$ and equal to one at $y_0$. This gives a horizontal isotopy $(k_t)_{0\leq t\leq1}$ supported in $J_*\times V$.
It fixes $J_*\times S$ pointwise and satisfies $k_1(a,y_0)=(x_0,y_0)$.

For each sufficiently narrow interval $J$ with $I\subset J\subset J_*$, choose an increasing circle homeomorphism $\chi_J$ which is the
identity on $I$, maps $J_*$ onto $J$, and is linear on the two collars. We also write $\chi_J$ for the torus map
$(x,y)\mapsto(\chi_J(x),y)$, and put
\begin{equation}
 k_{J,t}=\chi_J k_t\chi_J^{-1}.
\end{equation}
On $J_*$ the maps $\chi_J$ are Lipschitz with constant at most one. We also have $u\circ\chi_J=u$ there, because $\chi_J$ fixes
$I$ and $u$ is constant on each collar. Every compressed isotopy fixes $X$, and
$k_{J,1}(a,y_0)=(x_0,y_0)$ independently of the collar width.

We now vary the isotopy parameter slowly between successive iterates of the annulus. We first bound the allowed
change in this parameter independently of the number of iterates. We then choose the number of iterates and narrow $J$.

The intervals $f_0^j(I)$ are pairwise disjoint, so $|f_0^j(I)|\to0$ as $|j|\to\infty$.
Choose $L$ such that $|f_0^j(I)|<\varepsilon/4$ for $|j|>L$.

For any prescribed finite $N>L+2$, we can choose $J$ narrow enough that the intervals $f_0^j(J)$, $-N-2\leq
j\leq N+2$, have disjoint closures. Require also that $f_0^j(J)$ avoids $I^0$ and $I^1$ when $j\ne0$, and that
\begin{equation}
 |f_0^j(J)|<\varepsilon/2\quad (L<|j|\leq N+2),
 \qquad \operatorname{diam}\psi(J)<\eta.
\end{equation}
The avoidance condition gives the following formulas on $J\times\T^1$:
\begin{equation}\label{eq:tower-single-shear}
 g^j(x,y)=
 \begin{cases}
 (f_0^j(x),y+u(x)\bmod1),&1\leq j\leq N+2,\\
 (f_0^j(x),y),&-N-2\leq j\leq0.
 \end{cases}
\end{equation}
The omitted accumulated terms are integers and therefore vanish on the torus.

There is a number $\gamma>0$, independent of $N$ and of the subsequent collar compression, such that whenever
$|s-t|<\gamma$,
\begin{equation}\label{eq:tower-time-modulus}
 \sup_{z\in J\times\T^1}
 d\bigl(g^j k_{J,s}(z),g^j k_{J,t}(z)\bigr)<\varepsilon
 \qquad (-N-2\leq j\leq N+2).
\end{equation}
Indeed, for $|j|>L$ the horizontal coordinates lie in the interval $f_0^j(J)$, of length less than $\varepsilon/2$.  For
the remaining finitely many $j$, use uniform continuity of $f_0^j$ and the common Lipschitz bound for $\chi_J$.
This gives a bound on $|s-t|$ valid for all these horizontal coordinates. For the vertical estimate, write
$w=(x,y)\in J_*\times\T^1$ and $z=\chi_J(w)$. Since $k_t$ changes only the first coordinate,
\eqref{eq:tower-single-shear} and $u\circ\chi_J=u$ show that the second coordinate of $g^j k_{J,t}(z)$ is $y$ for
$j\leq0$, and $y+u(\pr_1(k_t(w)))\bmod1$ for $j\geq1$. Uniform continuity gives a bound on $|s-t|$
independent of $j$, $J$ and $N$. In the maximum product metric, these two coordinate estimates give
\eqref{eq:tower-time-modulus}.

Choose $N$ sufficiently large and numbers $t_j\in[0,1]$ satisfying
\begin{equation}
 t_0=1,\qquad t_j=0\ (|j|\geq N+1),\qquad
 |t_{j+1}-t_j|<\gamma.
\end{equation}
Only now choose $J$ with the preceding disjointness, avoidance and size properties and, in addition,
\begin{equation}\label{eq:relative-collar-smallness}
 (2N+1)\lambda(J)<\delta,
 \qquad (2N+1)\operatorname{Leb}_{\T^1}(\psi(J))<\delta.
\end{equation}
This is possible because $\lambda(I)=0$ and $\psi(I)$ is a point. Define $P$ on $f_0^j(J)\times\T^1$ by
\begin{equation}
 P=g^j k_{J,t_j}g^{-j}\qquad (-N\leq j\leq N),
\end{equation}
and define it to be the identity elsewhere.  The supports are disjoint, so this defines a homeomorphism. Replacing each
$t_j$ by $s t_j$, $0\leq s\leq1$, gives an isotopy from the identity to $P$. Every piece fixes $X$ because $g(X)=X$. Thus
the isotopy is relative to $X$.

For $z\in f_0^j(J)\times\T^1$, put $w=k_{J,t_j}^{-1}g^{-j}(z)$. To compare $PgP^{-1}(z)$ and $g(z)$,
it is enough to compare
\begin{equation}
  g^{j+1}k_{J,t_{j+1}}(w)\quad\hbox{and}\quad
  g^{j+1}k_{J,t_j}(w).
\end{equation}
Equation \eqref{eq:tower-time-modulus} gives the required forward estimate. The same argument with $j-1$ gives the
inverse estimate. At the first and last annuli, take the adjacent parameter to be zero. Outside the annuli
$f_0^j(J)\times\T^1$ with $-N-1\leq j\leq N+1$, the conjugated maps agree with $g$ and $g^{-1}$, respectively.

The map $P$ preserves each annulus $f_0^j(J)\times\T^1$ setwise. The oscillation of $h_0$ on this annulus is that of
$\psi$ on $J$, since $\psi\circ f_0^j=R_{j\alpha}\circ\psi$. This proves the factor estimate. On the central annulus,
$P=k_{J,1}$, so $z=(a,y_0)$ has $P(z)=(x_0,y_0)\in B$.
Finally, invariance of $\mu_0$ and the support inclusion give
\begin{equation}\label{eq:base-support-control}
 \mu_0(\operatorname{supp}P)\leq(2N+1)\lambda(J),\qquad
 \operatorname{Leb}_{\T^1}\bigl(h_0(\operatorname{supp}P)\bigr)
 \leq(2N+1)\operatorname{Leb}_{\T^1}(\psi(J)).
\end{equation}
Equation~\eqref{eq:relative-collar-smallness} gives both support bounds.
\end{proof}

The next two lemmas apply in both homotopy classes. The first proves that limiting fibres are annular continua.
The second shows that almost every circle fibre of $h_0$ survives the perturbations.

\begin{lemma}\label{lem:annular-fibre-limit}
Let $\psi:\T^1\to\T^1$ be a monotone degree-one map, and let $(H_n)$ be a sequence of torus homeomorphisms
isotopic to the identity.
\begin{enumerate}
\item Define $h_0(x,y)=\psi(x)$. If $h_0\circ H_n^{-1}$ converges uniformly to $h:\T^2\to\T^1$, then every fibre
$h^{-1}(t)$ is compact and connected, and its complement in $\T^2$ is an open annulus.
\item Define $h_0(x,y)=\psi(y)$. If $h_0\circ H_n^{-1}$ converges uniformly to $h:\T^2\to\T^1$, then every fibre
$h^{-1}(t)$ is compact and connected, and its complement in $\T^2$ is an open annulus.
\end{enumerate}
\end{lemma}

\begin{proof}
We prove the statement for $h_0(x,y)=\psi(x)$. The other case follows by exchanging the coordinates.
On the cyclic cover $\R\times\T^1$, choose real lifts $p_n$ of $h_0H_n^{-1}$ and $p$ of
$h$ with $p_n\to p$ uniformly. Such a normalisation is possible because the circle-valued maps converge uniformly
and have the same homotopy class. The map $p$ is proper, since $p(x,y)-x$ is bounded and periodic.
Each $p_n$ is obtained by precomposing $\widetilde\psi\circ\pr_1$ with an annular homeomorphism, up to an integer
translation. Thus inverse images of compact intervals under $p_n$ are connected.

Given $a\leq b$, choose $\varepsilon_n\downarrow0$ with $\varepsilon_n>\|p_n-p\|_{C^0}$. The connected compact sets
\begin{equation}\label{eq:interval-preimage-limit}
 p_n^{-1}([a-\varepsilon_n,b+\varepsilon_n])
 \longrightarrow p^{-1}([a,b])
 \quad\text{in the Hausdorff metric}.
\end{equation}
Indeed, these sets contain $p^{-1}([a,b])$ and lie in a common compact strip. Every limit of points in them has its
$p$-value in $[a,b]$. Hence $p^{-1}([a,b])$ is connected. Exhaustion by compact intervals now shows that the inverse
image of every open interval or ray is connected. In particular, $p^{-1}(t)$ is a compact continuum. Its complement in the
cylinder has exactly two components, $\{p<t\}$ and $\{p>t\}$. These contain the two ends.

The identity $p(x+1,y)=p(x,y)+1$ shows that projection embeds $p^{-1}(t)$ as the full fibre
$h^{-1}(t\bmod1)$. Its complement in the torus identifies with $p^{-1}((t,t+1))$. Cap the two ends of the cylinder
to obtain a sphere. The complement of this last region consists of the disjoint continua
$\{p\leq t\}$ together with the lower end, and $\{p\geq t+1\}$ together with the upper end. Each is
nonseparating: its complement is the connected inverse image of a ray together with the opposite end.

Each of these continua has decreasing closed Jordan-disc neighbourhoods. Its complementary
domain in the sphere is simply connected. A Riemann map onto this domain sends concentric circles to Jordan
curves. The closed discs on the side containing the continuum decrease to it as the radii tend to one. Choose the
two families disjoint and strictly nested. The complements of their interiors exhaust $p^{-1}((t,t+1))$ by closed annuli.
Successive applications of the Schoenflies theorem between their boundary curves identify the exhaustion with that
of an open annulus. 
\end{proof}

\begin{lemma}\label{lem:jordan-fibre-survival}
Let $h_0:\T^2\to\T^1$ have fibres which are essential Jordan curves for every parameter outside a countable set $Z$. Suppose that
$H_0=\operatorname{id}$, $H_n=H_{n-1}Q_n\in\operatorname{Homeo}_0(\T^2)$, and $h_0H_n^{-1}\to h$ uniformly.
Assume that $h$ is an irrational circle factor, homotopic to a coordinate projection, whose fibres are thin
essential annular continua. If
\begin{equation}\label{eq:summable-source-supports}
 \sum_{n\geq1}\operatorname{Leb}_{\T^1}\bigl(h_0(\operatorname{supp}Q_n)\bigr)<\infty,
\end{equation}
then almost every fibre of $h$ is a Jordan curve.
\end{lemma}

\begin{proof}
Put $D_n=h_0(\operatorname{supp}Q_n)$. Lemma~\ref{lem:measure-facts}\eqref{item:summable-sets} gives
\begin{equation}\label{eq:null-support-limsup}
 D_\infty=\limsup_{n\to\infty}D_n,
 \qquad \operatorname{Leb}_{\T^1}(D_\infty)=0.
\end{equation}
Fix $\theta\notin Z\cup D_\infty$. For some $n_0$, every $Q_n$ with $n>n_0$ fixes the circle
$h_0^{-1}(\theta)$ pointwise. Since $H_n=H_{n-1}Q_n$, it follows that
\begin{equation}\label{eq:surviving-source-circle}
 H_n(h_0^{-1}(\theta))=H_{n_0}(h_0^{-1}(\theta))
 \qquad(n\geq n_0).
\end{equation}
The factors $h_0H_n^{-1}$ equal $\theta$ on this fixed essential Jordan curve. Uniform convergence places the curve
in $h^{-1}(\theta)$. The curve is a circloid. It therefore equals the unique circloid core of the thin fibre $h^{-1}(\theta)$. By
Proposition~\ref{prop:countable-spikes}, only countably many fibres differ from their cores. Apart from their parameters
and $Z\cup D_\infty$, every final fibre is exactly the surviving circle. The exceptional set is null.
\end{proof}

\begin{proof}[Proof of Theorem~\ref{thm:thin-counterexample}]
\textbf{The space of conjugates.}
Use the maps and sets fixed above for the prescribed vector $\rho=(\alpha,\beta)$.
The invariant measure $\mu_0$ has support $K\times\T^1$ and satisfies
$(h_0)_*\mu_0=\operatorname{Leb}_{\T^1}$. For every torus homeomorphism $H$ isotopic to the identity relative to $X$, put
\begin{equation}
 f_H=HgH^{-1},\qquad h_H=h_0H^{-1},\qquad \mu_H=H_*\mu_0.
\end{equation}
Let $\mathcal X$ be the closure of these triples in
\begin{equation}
 \operatorname{Homeo}_0(\T^2)\times
 C_{[\pr_1]}(\T^2,\T^1)\times\mathcal P(\T^2),
\end{equation}
where $\mathcal P(\T^2)$ is the compact metrisable space of Borel probability measures with the weak topology. The
first factor has the metric $d_{\mathrm H}$ from \eqref{eq:homeomorphism-metric}. Here
$C_{[\pr_1]}(\T^2,\T^1)$ denotes the closed homotopy class of $\pr_1$ in the uniform metric. Thus $\mathcal X$ is a
nonempty complete metric space. Every
$(f,h,\mu)\in\mathcal X$ satisfies
\begin{equation}\label{eq:closure-identities}
 hf=R_\alpha h,\qquad f|_X=g|_X,\qquad h|_X=h_0|_X.
\end{equation}
The map $h$ is onto because its homotopy class is nontrivial. Lemma~\ref{lem:measure-facts}\eqref{item:measure-limits} also gives
$f_*\mu=\mu$ and $h_*\mu=\operatorname{Leb}_{\T^1}$.

The fixed dynamics on $X$ determines the rotation vector of every map in this closure. Choose the lift $G$ from
Proposition~\ref{prop:denjoy-model} and a point $z_*\in\pi^{-1}(X)$. A relative isotopy defining $H$ lifts to an
isotopy fixing $\pi^{-1}(X)$ pointwise. Hence
\begin{equation}\label{eq:protected-conjugate-lift}
 \widetilde H G\widetilde H^{-1}=G
 \quad\text{on }\pi^{-1}(X).
\end{equation}
For every limiting triple, normalise the lift $F$ by $F(z_*)=G(z_*)$.
Normalised lifts depend continuously on $f$, so $F=G$ on $\pi^{-1}(X)$. This set is invariant under both lifts.
Hence $F^n(z_*)=G^n(z_*)$ for every $n$. The orbit of $z_*$ therefore has rotation vector $\rho$.
Lemma~\ref{lem:rotation-singleton}, applied to the factor identity in \eqref{eq:closure-identities}, makes $\rot(F)$
a singleton. The orbit of $z_*$ identifies it as $\{\rho\}$.

\textbf{The full-support condition.}
Let $(B_n)_{n\geq1}$ be a countable basis of nonempty connected open balls in the torus. Let $\mathcal W_n$ be the set of triples
$(f,h,\mu)\in\mathcal X$ for which $\mu(B_n)>0$. The mass of an open set is lower
semicontinuous in the weak topology. Thus this condition is open. We next prove density.

It suffices to perturb a conjugate triple $(f_H,h_H,\mu_H)$. Its measure has support $H(K\times\T^1)$.
If $B_n$ meets this set, it already has positive measure. Otherwise the connected set $H^{-1}(B_n)$ lies in the
interior of one wandering annulus $g^m(A)$, where $A$ is either $I^0\times\T^1$ or $I^1\times\T^1$. Apply
Lemma~\ref{lem:relative-tower} with target
\begin{equation}
 g^{-m}H^{-1}(B_n)\subset\operatorname{int}A,
\end{equation}
and put $Q=g^mPg^{-m}$ and $H'=HQ$. Since $g^m(X)=X$, both $Q$ and $H'$ are isotopic to the identity
relative to $X$.
Conjugating the forward and inverse estimates by the fixed map $Hg^m$ makes $d_{\mathrm H}(f_{H'},f_H)$ as small
as desired. The factor estimate is unchanged by rotation of the base, since $h_0g^m=R_{m\alpha}h_0$.

Invariance of $\mu_0$ gives $\mu_0(\operatorname{supp}Q)=\mu_0(\operatorname{supp}P)<\delta$. Consequently, for
any continuous test function $\varphi$,
\begin{equation}\label{eq:measure-perturbation}
 \left|\int\varphi\,d\mu_{H'}-\int\varphi\,d\mu_H\right|
 \leq 2\|\varphi\|_\infty\mu_0(\operatorname{supp}Q)
 \leq2\|\varphi\|_\infty\delta.
\end{equation}
Thus the whole triple can be kept arbitrarily close to the original one. The set of parameters $h_0(\operatorname{supp}Q)$ also satisfies
\begin{equation}\label{eq:translated-base-support}
 \operatorname{Leb}_{\T^1}\bigl(h_0(\operatorname{supp}Q)\bigr)
 =\operatorname{Leb}_{\T^1}\bigl(h_0(\operatorname{supp}P)\bigr)<\delta.
\end{equation}

The point $z\in\partial A$ supplied by the lemma belongs to $\operatorname{supp}\mu_0$. Hence
$g^m z\in\operatorname{supp}\mu_0$ and $H'(g^m z)=Hg^mPz\in B_n$. This puts a support point of $\mu_{H'}$ in
$B_n$, so $\mu_{H'}(B_n)>0$. Positivity is an open condition. Thus $\mathcal W_n$ is dense, and the perturbation
can be chosen with $\operatorname{Leb}_{\T^1}(h_0(\operatorname{supp}Q))$ arbitrarily small.

\textbf{The limiting map.}
To retain the summable support bounds, we choose nested balls explicitly. Starting with $H_0=\operatorname{id}$,
choose $H_n=H_{n-1}Q_n$ and a closed ball $\mathcal B_n$ centred at its conjugate triple such that
\begin{equation}\label{eq:nested-baire-balls}
 \mathcal B_n\subset\operatorname{int}\mathcal B_{n-1}
 \cap\bigcap_{j=1}^n\mathcal W_j,
 \quad \operatorname{diam}\mathcal B_n<2^{-n},
 \quad
 \operatorname{Leb}_{\T^1}\bigl(h_0(\operatorname{supp}Q_n)\bigr)<2^{-n}.
\end{equation}
For $n=1$, omit $\mathcal B_0$. If the preceding triple already belongs to $\mathcal W_n$, take
$Q_n=\operatorname{id}$. Otherwise use the density argument with the last bound. Keep the perturbed centre inside
the preceding ball and the earlier open conditions, then choose $\mathcal B_n$ small enough to preserve them.
Completeness gives a unique limit
\begin{equation}\label{eq:chosen-baire-limit}
 (f,h,\mu)\in\bigcap_{n\geq1}\mathcal B_n
 \subset\bigcap_{n\geq1}\mathcal W_n.
\end{equation}
The invariant measure $\mu$ has full support. Lemma~\ref{lem:measure-facts}\eqref{item:measure-recurrence}
therefore makes $f$ nonwandering. Since its rotation vector is $\rho$, Lemma~\ref{lem:potrie-transitivity} gives topological
transitivity. The equality $f|_X=g|_X$ preserves the pairwise disjoint
uniquely ergodic minimal Cantor sets $K\times\{s\}$, $s\in S$.

\textbf{The fibres and area preservation.}
Since $h_*\mu=\operatorname{Leb}_{\T^1}$, every fibre has $\mu$-measure zero. Full support of $\mu$ therefore
forces every fibre to have empty interior. Lemma~\ref{lem:annular-fibre-limit} gives its annular topology.
The fibres of $h_0$ are essential Jordan curves away from the two countable blown-up orbits. Equation
\eqref{eq:nested-baire-balls} and Lemma~\ref{lem:jordan-fibre-survival} now show that almost every fibre of $h$ is
a Jordan curve.

Finally, $\mu$ is nonatomic, since $\mu(\{z\})\leq\mu(h^{-1}(h(z)))=\operatorname{Leb}_{\T^1}(\{h(z)\})=0$. Lemma~\ref{lem:oxtoby-ulam} therefore gives a homeomorphism
$\Phi\in\operatorname{Homeo}_0(\T^2)$ with $\Phi_*\mu=\operatorname{Leb}_{\T^2}$. Replace $(f,h,\mu)$ by
$(\Phi f\Phi^{-1},h\Phi^{-1},\Phi_*\mu)$ and retain the notation $f,h,\mu$.
The map $f$ is now area-preserving and remains transitive. Its pairwise disjoint uniquely ergodic minimal Cantor
sets include $\Phi(K\times\{s\})$, $s\in S$. A lift of $\Phi$ differs from the identity by a bounded function. Thus this conjugacy preserves the rotation vector. It also preserves the homotopy class of the factor
and the fibre type at each base parameter. This proves the theorem.
\end{proof}

\begin{remark}\label{rem:conjugacy-limits}
Before the final change of coordinates, the limiting triple can be chosen arbitrarily close to $(g,h_0,\mu_0)$.
This includes uniform closeness of the maps and their inverses. The conjugacies and their inverses need not converge.
\end{remark}

\section{The Dehn-twist construction}\label{sec:dehn}

Throughout this section, fix $k\in\Z\setminus\{0\}$ and $\alpha\in\R\setminus\Q$ as in
Theorem~\ref{thm:dehn-counterexample}. We place the Denjoy dynamics in the second coordinate. We use a circle map of degree $k$ for the horizontal displacement. We write $g$ for the initial model, with lift $G$, and $f$ for the map obtained after taking
the limit and changing coordinates.

\begin{lemma}\label{lem:factor-bounded-deviation}
Let $f\in\operatorname{Homeo}_k(\T^2)$, and suppose that a continuous map $h:\T^2\to\T^1$, homotopic to $\pr_2$,
satisfies $h\circ f=R_\alpha\circ h$.  There is a lift $F$ of $f$ such that
\begin{equation}\label{eq:factor-lift-relation}
 \sup_{z\in\R^2,\,n\in\Z}
 \left|\pr_2(F^n(z)-z)-n\alpha\right|<\infty.
\end{equation}
In particular, the vertical rotation set of $F$ is $\{\alpha\}$.
\end{lemma}

\begin{proof}
Choose a real lift $H:\R^2\to\R$ of $h$. Choose a lift $F$ of $f$ and adjust it by a deck transformation so that $H\circ
F=H+\alpha$.  Since $h$ is homotopic to $\pr_2$, there is a bounded $\Z^2$-periodic function $\varphi$ such that
$H(x,y)=y+\varphi(x,y)$.  Iteration gives
\begin{equation}\label{eq:factor-deviation-coboundary}
 \pr_2(F^n(z)-z)-n\alpha
 =\varphi(z)-\varphi(F^n(z))
 \qquad(n\in\Z).
\end{equation}
The absolute value is at most $2\|\varphi\|_\infty$.
\end{proof}

For rational vertical rotation, Addas-Zanata, Tal and Garcia proved that a singleton rotation set implies bounded
deviations \cite[Theorem~2]{AddasZanataTalGarcia2014}. For an irrational vertical rotation number in a nontrivial Dehn-twist
class, bounded deviations imply the existence of a circle factor \cite[Theorem~A]{Correa2026}.

\begin{proposition}\label{prop:dehn-denjoy-model}
There exist an orientation-preserving Denjoy homeomorphism $f_0:\T^1\to\T^1$ with unique minimal Cantor set $K$,
a monotone degree-one map $\psi:\T^1\to\T^1$, and a degree-$k$ map $\tau:\T^1\to\T^1$, together with a
closed wandering interval $I=[a,b]$ for $f_0$, such that the maps
\begin{equation}\label{eq:dehn-source-map}
 g(x,y)=(x+\tau(y),f_0(y)),\qquad h_0(x,y)=\psi(y),
\end{equation}
satisfy $h_0\circ g=R_\alpha\circ h_0$. The map $f_0$ is uniquely ergodic, with invariant probability measure
$\lambda$, and $\tau|_{\T^1\setminus\operatorname{int}I}=0$. The minimal sets of $g$ are exactly
\begin{equation}\label{eq:dehn-source-minimal-sets}
 \{s\}\times K,\qquad s\in\T^1.
\end{equation}
Each of these minimal sets is uniquely ergodic. Moreover, $g\in\operatorname{Homeo}_k(\T^2)$ and a lift of $g$ has
$\alpha$-bounded vertical deviations.
\end{proposition}

\begin{proof}
Perform a Denjoy blow-up of one orbit of $R_\alpha$.  This gives $f_0$, $\psi$ and a closed wandering interval $I=[a,b]$
such that
\begin{equation}\label{eq:dehn-denjoy-factor}
 \psi\circ f_0=R_\alpha\circ\psi.
\end{equation}
The intervals $f_0^j(I)$ are the closures of the complementary intervals of $K$.  As in
Proposition~\ref{prop:denjoy-model}, the map $f_0$ is uniquely ergodic.  Write its invariant probability measure as
$\lambda$.

Choose a circle coordinate in which a lift of $I$ is contained in $(0,1)$. On $[0,1]$, set $\widetilde\tau=0$ to the left of $I$
and $\widetilde\tau=k$ to the right of $I$. Interpolate linearly from $0$ to $k$ on $I$. Extend it by
\begin{equation}\label{eq:degree-k-cocycle-lift}
 \widetilde\tau(t+n)=\widetilde\tau(t)+kn
 \qquad(t\in\R,\ n\in\Z).
\end{equation}
Its projection is a continuous degree-$k$ circle map $\tau$.  It vanishes on $K$ and on every complementary interval
other than $\operatorname{int}I$.

Let $F_0$ be the lift of $f_0$ with rotation number $\alpha$. Then
\begin{equation}\label{eq:dehn-source-lift}
 G(x,y)=(x+\widetilde\tau(y),F_0(y))
\end{equation}
satisfies
\begin{equation}\label{eq:dehn-equivariance}
 G(z+(m,n))=G(z)+A_k(m,n).
\end{equation}
Thus $g\in\operatorname{Homeo}_k(\T^2)$.  Equation \eqref{eq:dehn-denjoy-factor} gives the factor identity, and
Lemma~\ref{lem:factor-bounded-deviation} gives the deviation bound.

On $\T^1\times K$ we have $g(x,y)=(x,f_0(y))$. Hence every set in \eqref{eq:dehn-source-minimal-sets} is a uniquely
ergodic minimal Cantor set. Conversely, the second-coordinate projection of a minimal set is $K$. The first
coordinate is invariant on $\T^1\times K$. This gives all the minimal sets of $g$.
\end{proof}

\begin{lemma}\label{lem:dehn-relative-tower}
Fix the maps of Proposition~\ref{prop:dehn-denjoy-model} and a Cantor set $S\subset\T^1$. Put
\begin{equation}\label{eq:dehn-local-data}
 X=S\times K,\qquad \mu_0=\operatorname{Leb}_{\T^1}\times\lambda,
 \qquad A=\T^1\times I.
\end{equation}
For every $\varepsilon,\eta,\delta>0$ and every nonempty open set $B\subset\operatorname{int}A$, there are a
homeomorphism $P$, isotopic to the identity relative to $X$, and a point $z\in\partial A\cap(\T^1\times K)$ such that
\begin{align}
 d_0(PgP^{-1},g)&<\varepsilon,
 &d_0(Pg^{-1}P^{-1},g^{-1})&<\varepsilon,
 \label{eq:dehn-local-dynamics}\\
 d_0(h_0P^{-1},h_0)&<\eta,
 &P(z)&\in B,
 \label{eq:dehn-local-factor}\\
 \mu_0(\operatorname{supp}P)&<\delta,
 &\operatorname{Leb}_{\T^1}\bigl(h_0(\operatorname{supp}P)\bigr)&<\delta.
 \label{eq:dehn-local-support}
\end{align}
\end{lemma}

\begin{proof}
Choose a slightly larger interval $J_*\supset I$. Since $S$ is nowhere dense, we may choose $(x_0,y_0)\in B$ with
$x_0\notin S$. Choose a vertical isotopy $(\kappa_t)_{0\leq t\leq1}$ supported in $\T^1\times J_*$. It can be chosen to fix $S\times J_*$
pointwise and satisfy
\begin{equation}\label{eq:dehn-boundary-finger}
 \kappa_1(x_0,a)=(x_0,y_0),
\end{equation}
where $a$ is one endpoint of $I$. To construct it, take a small horizontal interval about $x_0$ disjoint from $S$. Move $a$ to
$y_0$ there. Multiply the isotopy parameter by a horizontal cutoff.

For every sufficiently narrow interval $J$ with $I\subset J\subset J_*$, choose an increasing circle homeomorphism $\chi_J$ which is the
identity on $I$, maps $J_*$ onto $J$, and is linear on the two collars. We use the same symbol for
$(x,y)\mapsto(x,\chi_J(y))$ and put
\begin{equation}\label{eq:dehn-compressed-isotopy}
 \kappa_{J,t}=\chi_J\kappa_t\chi_J^{-1}.
\end{equation}
The maps $\chi_J$ are uniformly Lipschitz on $J_*$, with constant at most one. We also have $\tau\circ\chi_J=\tau$ on
$J_*$ as circle-valued maps, because $\chi_J$ fixes $I$ and $\tau$ is zero on both collars.

Choose $L$ so large that $|f_0^j(I)|<\varepsilon/4$ whenever $|j|>L$. For any prescribed $N>L+2$, we can choose $J$
narrow enough that the intervals $f_0^j(J)$, $|j|\leq N+2$, have pairwise disjoint closures. Require also that $f_0^j(J)$ avoids $I$
for $j\neq0$, and that $|f_0^j(J)|<\varepsilon/2$ for $L<|j|\leq N+2$. On $\T^1\times J$ we then have
\begin{equation}\label{eq:dehn-tower-single-shear}
 g^j(x,y)=
 \begin{cases}
  (x+\tau(y),f_0^j(y)),&j\geq1,\\
  (x,f_0^j(y)),&j\leq0,
 \end{cases}
\end{equation}
for $-N-2\leq j\leq N+2$.

There is $\gamma>0$, independent of $N$ and the collar compression, such that
$|s-t|<\gamma$ implies
\begin{equation}\label{eq:dehn-tower-time-modulus}
 \sup_{w\in\T^1\times J}
 d\bigl(g^j\kappa_{J,s}(w),g^j\kappa_{J,t}(w)\bigr)<\varepsilon
\end{equation}
for $-N-2\leq j\leq N+2$. To see this, write $u=(x,y)\in\T^1\times J_*$. Since $\kappa_t$ changes only the
second coordinate, \eqref{eq:dehn-tower-single-shear} and $\tau\circ\chi_J=\tau$ show that the first coordinate of
$g^j\kappa_{J,t}(\chi_J(u))$ is $x$ for $j\leq0$, and $x+\tau(\pr_2(\kappa_t(u)))$ for $j\geq1$.
Uniform continuity gives a bound on $|s-t|$ independent of $j$, $J$ and $N$. For $|j|>L$, the two second coordinates
lie in the same interval $f_0^j(J)$ of length less than $\varepsilon/2$. For the remaining finitely many indices,
use uniform continuity and the common Lipschitz bound for $\chi_J$.

Choose $N>L+2$ sufficiently large and numbers $t_j\in[0,1]$ such that
\begin{equation}
 t_0=1,\qquad t_j=0\quad(|j|\geq N+1),\qquad |t_{j+1}-t_j|<\gamma.
\end{equation}
Only now choose $J$ with the preceding disjointness, avoidance and size properties and, in addition,
\begin{equation}\label{eq:dehn-collar-smallness}
 \operatorname{diam}\psi(J)<\eta,
 \qquad (2N+1)\lambda(J)<\delta,
 \qquad (2N+1)\operatorname{Leb}_{\T^1}(\psi(J))<\delta.
\end{equation}
This is possible because $\lambda(I)=0$ and $\psi(I)$ is a point. Define
\begin{equation}\label{eq:dehn-tower-perturbation}
 P=g^j\kappa_{J,t_j}g^{-j}
 \quad\hbox{on }\T^1\times f_0^j(J),\quad -N\leq j\leq N,
\end{equation}
and let $P$ be the identity elsewhere. The annuli are disjoint, so $P$ is a homeomorphism. Scaling the parameters
$t_j$ gives an isotopy relative to $X$. The comparison of adjacent parameters in \eqref{eq:dehn-tower-time-modulus}
proves \eqref{eq:dehn-local-dynamics}. As in Lemma~\ref{lem:relative-tower}, take $t_j=0$ on the adjacent annuli
with $|j|=N+1$ when checking the forward and inverse estimates. The map $P$ preserves each annulus setwise. The
oscillation of $h_0$ on each annulus is $\operatorname{diam}\psi(J)$. This proves the first part of \eqref{eq:dehn-local-factor}.
Equations \eqref{eq:dehn-boundary-finger} and \eqref{eq:dehn-compressed-isotopy} give $P(x_0,a)=(x_0,y_0)\in B$.
Finally, the support lies in the $2N+1$ annuli $\T^1\times f_0^j(J)$, $-N\leq j\leq N$. Each has
$\mu_0$-measure $\lambda(J)$, and its image under
$h_0$ is a rotate of $\psi(J)$. The last two inequalities in \eqref{eq:dehn-collar-smallness} therefore give
\eqref{eq:dehn-local-support}.
\end{proof}

\begin{lemma}\label{lem:dehn-transitivity}
Let $k\in\Z\setminus\{0\}$ and $\alpha\in\R\setminus\Q$. Suppose that
$f\in\operatorname{Homeo}_k(\T^2)$ admits a continuous circle factor $h\simeq\pr_2$ satisfying
$h\circ f=R_\alpha\circ h$ and preserves a probability measure of full support. Then $f$ is topologically transitive.
\end{lemma}

\begin{proof}
The factor excludes periodic points. Lemma~\ref{lem:measure-facts}\eqref{item:measure-recurrence} makes every point
nonwandering. We claim that $f$ is not eventually annular. Suppose instead that an iterate $f^q$ were annular.
There would be a lift $G$ of $f^q$ and a nonzero vector $v=(v_1,v_2)\in\Z^2$ for which
$\langle G^n(z)-z,v\rangle$ is uniformly bounded.
Equivariance gives
\begin{equation}\label{eq:dehn-deck-shear}
 G^n(z+(0,1))-G^n(z)=(nkq,1).
\end{equation}
Applying the annular bound at $z$ and $z+(0,1)$ forces $v_1=0$.  Thus the vertical displacement of a lift of $f^q$ is
bounded. On the other hand, every lift of $f^q$ differs from the $q$th iterate of a factor-compatible lift by an
integer deck translation. Since $A_{kq}$ fixes the second coordinate, Lemma~\ref{lem:factor-bounded-deviation} shows
that its vertical rotation number is $q\alpha-p$ for some $p\in\Z$. This cannot vanish because $\alpha$ is irrational.

Lemma~\ref{lem:fully-essential-recurrence} now applies. Thus the full orbit unions of
any two nonempty open sets $U,V$ contain fully essential domains. Such domains intersect, by the algebraic
intersection form on the torus. Hence $f^r(U)\cap V\neq\varnothing$ for some $r\in\Z$.

If $r\leq0$, put $m=-r$. Then $U\cap f^m(V)$ contains a nonempty open set $W$. Lemma~\ref{lem:measure-facts}\eqref{item:measure-recurrence}, applied to $W$, gives arbitrarily large $q$ with
$f^q(W)\cap W\neq\varnothing$.  Taking $q>m$ yields
$f^{q-m}(U)\cap V\neq\varnothing$.  Thus the iterate can always be chosen positive.
\end{proof}

\begin{proof}[Proof of Theorem~\ref{thm:dehn-counterexample}]
With $k$ and $\alpha$ fixed as above, let $g,h_0,f_0,K,\lambda$ and $I$ be as in
Proposition~\ref{prop:dehn-denjoy-model}. Thus $g(x,y)=(x+\tau(y),f_0(y))$ and $h_0(x,y)=\psi(y)$.
Choose a Cantor set $S\subset\T^1$, and put
\begin{equation}\label{eq:dehn-protected-set-measure}
 X=S\times K,
 \qquad
 \mu_0=\operatorname{Leb}_{\T^1}\times\lambda.
\end{equation}
The set $X$ is $g$-invariant. The measure $\mu_0$ is an invariant nonatomic probability measure supported on $\T^1\times K$.
It satisfies $(h_0)_*\mu_0=\operatorname{Leb}_{\T^1}$.

For each torus homeomorphism $H$ isotopic to the identity relative to $X$, consider the triple
\begin{equation}\label{eq:dehn-conjugate-triple}
 (HgH^{-1},h_0H^{-1},H_*\mu_0).
\end{equation}
Take the closure of these triples in
\begin{equation}\label{eq:dehn-baire-space}
 \operatorname{Homeo}_k(\T^2)\times
 C_{[\pr_2]}(\T^2,\T^1)\times\mathcal P(\T^2),
\end{equation}
where $C_{[\pr_2]}(\T^2,\T^1)$ is the closed homotopy class of $\pr_2$. Using the metrics of Section~\ref{sec:preliminaries},
this is a complete space. Uniform convergence and Lemma~\ref{lem:measure-facts}\eqref{item:measure-limits} show that
every limiting triple $(f,h,\mu)$ satisfies
\begin{equation}\label{eq:dehn-closure-identities}
 h\circ f=R_\alpha\circ h,
 \qquad f|_X=g|_X,
 \qquad h|_X=h_0|_X,
 \qquad f_*\mu=\mu,
 \qquad h_*\mu=\operatorname{Leb}_{\T^1}.
\end{equation}
The map $h$ is onto because its homotopy class is nontrivial.

The full-support argument in the proof of Theorem~\ref{thm:thin-counterexample} applies with
Lemma~\ref{lem:dehn-relative-tower} in place of Lemma~\ref{lem:relative-tower}. Let $(B_n)$ be a countable basis of
nonempty connected open balls. If $B_n$ misses $H(\T^1\times K)$, then $H^{-1}(B_n)$ lies in one wandering
annulus $\T^1\times\operatorname{int}f_0^m(I)$. Apply Lemma~\ref{lem:dehn-relative-tower} with target
$g^{-m}H^{-1}(B_n)$, and put $Q=g^mPg^{-m}$ and $H'=HQ$. The map $H'$ sends a point of
$\operatorname{supp}\mu_0$ into $B_n$, so $(H')_*\mu_0(B_n)>0$.
The dynamical and factor estimates, together with \eqref{eq:measure-perturbation}, keep the conjugate triple
arbitrarily close to the original one. Invariance of $\mu_0$ and the identity $h_0g^m=R_{m\alpha}h_0$ give
the base-support bound \eqref{eq:translated-base-support}. Thus each condition $\mu(B_n)>0$ is open and dense,
and the perturbation can have arbitrarily small $\operatorname{Leb}_{\T^1}(h_0(\operatorname{supp}Q))$.

Choose nested balls as in \eqref{eq:nested-baire-balls}. At stage $n$, require the first $n$ full-support conditions.
Write the successive conjugacies as $H_n=H_{n-1}Q_n$. In addition, require
\begin{equation}\label{eq:dehn-summable-base-support}
 \operatorname{Leb}_{\T^1}
 \bigl(h_0(\operatorname{supp}Q_n)\bigr)<2^{-n}.
\end{equation}
The resulting limit $(f,h,\mu)$ satisfies \eqref{eq:dehn-closure-identities}, and $\mu$ has full support.
Since $h_*\mu=\operatorname{Leb}_{\T^1}$, every fibre has $\mu$-measure zero. Full support therefore forces
every fibre to have empty interior.

Since $h_0=\psi\circ\pr_2$, Lemma~\ref{lem:annular-fibre-limit} gives the essential annular topology of every
limiting fibre. The fibres of $h_0$ are Jordan curves away from the countable blown-up orbit. Equation
\eqref{eq:dehn-summable-base-support} and Lemma~\ref{lem:jordan-fibre-survival} therefore show that almost every
limiting fibre is a Jordan curve.

For every $s\in S$, the equality $f|_X=g|_X$ preserves the uniquely ergodic minimal Cantor set $\{s\}\times K$. 
These sets are pairwise disjoint. The zero measure of every fibre also shows that $\mu$ is nonatomic.
By Lemma~\ref{lem:oxtoby-ulam}, there is $\Phi\in\operatorname{Homeo}_0(\T^2)$ such that
$\Phi_*\mu=\operatorname{Leb}_{\T^2}$. Replace $(f,h,\mu)$ by
$(\Phi f\Phi^{-1},h\Phi^{-1},\Phi_*\mu)$ and retain the notation $f,h,\mu$.

The map $f$ now preserves area and remains in $\operatorname{Homeo}_k(\T^2)$. Its factor $h$ is homotopic to
$\pr_2$ and satisfies $h\circ f=R_\alpha\circ h$. Each fibre is the image under $\Phi$ of the corresponding
old fibre. Thus every fibre remains a thin annular continuum, and almost every fibre is still a Jordan curve. The sets
$\Phi(\{s\}\times K)$, $s\in S$, remain pairwise disjoint uniquely ergodic minimal Cantor sets.
Lemma~\ref{lem:factor-bounded-deviation}, applied after this change of coordinates, gives a lift $F$ with
$\alpha$-bounded vertical deviations. Lemma~\ref{lem:dehn-transitivity} gives topological transitivity.
\end{proof}

\section{Thin annular fibres}\label{sec:thin-geometry}

We describe the two sides and the unique circloid core of a thin annular fibre. Only countably many fibres are
larger than their cores. Outside this countable set, local connectedness, being a Jordan curve, and having a
Jordan-curve core are equivalent. This reduces Corollary~\ref{cor:locally-connected} to
Theorem~\ref{thm:generic-jordan-fibres}.

For a totally irrational pseudo-rotation, Lemma~\ref{prop:canonical-fibres} shows that every fibre is an essential annular continuum.

\subsection{The two sides of a thin fibre}

Throughout this section, assume that $h$ is homotopic to a coordinate projection and
$h\circ f=R_\alpha\circ h$, where $\alpha$ is irrational. Assume also that every fibre is a thin essential annular continuum.
The results below apply in both homotopy classes considered in the introduction.

We first choose the annular cover. Let $\Lambda_h=\ker(h_*:\pi_1(\T^2)\to\pi_1(\T^1))$.
Set $\mathbb A=\R^2/\Lambda_h$ and let $q:\mathbb A\to\T^2$ be the covering map. Thus $\mathbb A\simeq\R\times\T^1$.
The circle direction represents $\Lambda_h$. The real direction corresponds to the homotopy class of $h$.
These coordinates do not parametrise the individual fibres. The map $h$ has a lift $\widehat h:\mathbb A\to\R$ satisfying
\begin{equation}\label{eq:annular-lift-period}
 \widehat h(Tz)=\widehat h(z)+1
\end{equation}
for the remaining deck transformation $T$. The identity $h_*f_*=h_*$ shows that $f$ lifts to this cover.
We choose its lift $\widehat f$ so that $\widehat h\circ\widehat f=\widehat h+\alpha$.
The map $\widehat h$ differs from the real coordinate on $\mathbb A$ by a
bounded function. Hence it is proper. Put
\begin{equation}\label{eq:lifted-fibres}
 A_t=\widehat h^{-1}(t),\qquad t\in\R.
\end{equation}
The restriction $q|_{A_t}$ is a homeomorphism onto $h^{-1}(t\bmod1)$. Each $A_t$ is a compact essential annular
continuum in $\mathbb A$. Call the end where $\widehat h$ tends to $-\infty$ the lower end.
At the upper end, $\widehat h$ tends to $+\infty$. Since $A_t$ is an essential annular continuum, $\mathbb
A\setminus A_t$ has exactly two components. The continuous function $\widehat h-t$ does not vanish on either component.
It therefore has constant sign on each component. It is negative on the lower component and positive on the upper component. Hence
\begin{equation}\label{eq:fibre-complement-sides}
 U^-(A_t)=\{\widehat h<t\},
 \qquad U^+(A_t)=\{\widehat h>t\}.
\end{equation}
Define the boundaries of these two components and their common core by
\begin{equation}\label{eq:fibre-sides-core}
 A_t^-=\partial U^-(A_t),\qquad
 A_t^+=\partial U^+(A_t),\qquad
 C_t=A_t^-\cap A_t^+.
\end{equation}

Lemma~\ref{lem:thin-annular-sides}, applied to the proper lift $\widehat h$, gives
\begin{equation}\label{eq:thin-fibre-decomposition}
 A_t=A_t^-\cup A_t^+,
\end{equation}
and shows that $C_t=A_t^-\cap A_t^+$ is the unique circloid contained in $A_t$. The same lemma gives the one-sided
Hausdorff limits
\begin{align}
 \lim_{s\nearrow t}^{\mathcal H}A_s&=A_t^- ,
 \label{eq:left-Hausdorff}\\
 \lim_{s\searrow t}^{\mathcal H}A_s&=A_t^+ .
 \label{eq:right-Hausdorff}
\end{align}
The sets $A_t^-$ and $A_t^+$ are continua, since they are Hausdorff limits of the continua $A_s$.
The projections $q(A_t^\pm)$ and $q(C_t)$ depend only on $\theta=t\bmod 1$.
We denote these projections by $A_\theta^\pm$ and $C_\theta$.

The following proposition shows that all but countably many fibres coincide with their circloid cores.

\begin{proposition}\label{prop:countable-spikes}
Let
\begin{equation}\label{eq:exceptional-fibres}
 E=\{\theta\in\T^1:h^{-1}(\theta)\neq C_\theta\}.
\end{equation}
Then $E$ is countable and invariant under $R_\alpha$. Consequently, $h^{-1}(\T^1\setminus E)$ is an invariant
dense $G_\delta$ set of full measure for every $f$-invariant probability measure, and every fibre over
$\T^1\setminus E$ equals its core.
\end{proposition}

\begin{proof}
Suppose that $A_t\neq C_t$. By \eqref{eq:thin-fibre-decomposition}, there is a point in
$A_t^-\setminus C_t$ or $A_t^+\setminus C_t$. If $z\in A_t^-\setminus C_t$, then
$z\notin\overline{\{\widehat h>t\}}$. Hence $\widehat h\leq t$ near $z$, so $t$ is a local maximum value.
If $z\in A_t^+\setminus C_t$, the same argument gives $\widehat h\geq t$ near $z$, so $t$ is a local minimum value.

The local maximum values of a continuous real-valued function on a second-countable space form a countable set.
Indeed, fix a countable basis. Each local maximum value is the supremum of the function on some basic open set.
The same argument, with infima, applies to local minimum values. Thus only countably many real parameters
satisfy $A_t\neq C_t$. Passing modulo $\Z$ proves that $E$ is countable.

The lift $\widehat f$ maps $A_t^\pm$ to $A_{t+\alpha}^\pm$ and $C_t$ to $C_{t+\alpha}$.
Thus $E$ is invariant under $R_\alpha$. Each fibre is closed and has empty interior.
Since $E$ is countable, $h^{-1}(E)$ is meagre and $F_\sigma$. Its complement is therefore dense and $G_\delta$.
The complement is invariant, and every fibre over $\T^1\setminus E$ equals its core.
If $\mu$ is $f$-invariant, then $h_*\mu$ is invariant under the irrational rotation.
Hence $h_*\mu$ is Lebesgue measure. Since $E$ is countable, $\mu(h^{-1}(E))=0$.
\end{proof}

\subsection{Nonmeagre regular fibres}

\begin{lemma}\label{lem:lc-parameters}
Let $J$ be the set of parameters whose fibres are Jordan curves. Let $J_c$ be the set of parameters whose fibre cores are Jordan curves.
These two sets are Borel and invariant under $R_\alpha$. Moreover,
\begin{equation}\label{eq:regular-parameters-modulo-countable}
 J\subset J_c,
 \qquad J_c\setminus J\subset E,
\end{equation}
where $E$ is the countable set in Proposition~\ref{prop:countable-spikes}. Consequently, if $J_c$ is nonmeagre
or the fibres are locally connected on a nonmeagre set of parameters, then $J$ is residual.
\end{lemma}

\begin{proof}
We first show that the core of a thin locally connected annular continuum $A$ is a Jordan curve. Add one point at
each end of the annulus to obtain a sphere. The set $A$ separates these points, so
Lemma~\ref{lem:locally-connected-separator} gives a Jordan curve $C\subset A$. This curve separates the two ends.
Indeed, if one of the discs bounded by $C$ contained neither end and met the complement of $A$, that complement
would have a component containing neither end. If the disc did not meet the complement, it would lie in $A$,
contrary to thinness. Thus $C$ is essential. It is a circloid and hence is the unique core of $A$.

Put
\begin{equation}\label{eq:lc-parameters}
 L=\{\theta\in\T^1:h^{-1}(\theta)\text{ is locally connected}\}.
\end{equation}
Work in the space of nonempty subcontinua of $\T^2$, with the Hausdorff metric. For positive
integers $m,n$, let $\mathcal L_{m,n}$ consist of the continua $C$ such that, whenever $x,y\in C$ and $d(x,y)<1/n$,
some subcontinuum $D\subset C$ contains $x,y$ and satisfies $\operatorname{diam}D\leq1/m$.

The set $\mathcal L_{m,n}$ is closed. Suppose that $C_j\in\mathcal L_{m,n}$ tends to $C$.
Take $x,y\in C$ with $d(x,y)<1/n$. Choose $x_j,y_j\in C_j$ tending to $x,y$.
For all large $j$, there is a subcontinuum $D_j\subset C_j$ containing $x_j,y_j$ with diameter at most $1/m$.
The space of subcontinua is compact, so a subsequence of $D_j$ converges in the Hausdorff metric.
Its limit is a subcontinuum of $C$. It contains $x,y$ and has diameter at most $1/m$.

By the uniform characterisation of local connectedness for compact continua, the locally connected continua form
the Borel set
\begin{equation}\label{eq:lc-hyperspace-class}
 \bigcap_{m\geq1}\bigcup_{n\geq1}\mathcal L_{m,n}.
\end{equation}

The fibre map $\theta\mapsto h^{-1}(\theta)$ is Borel. To see this, let $U\subset\T^2$ be open.
The set of parameters whose fibres meet $U$ is $h(U)$. Since $U$ is a countable union of compact sets, $h(U)$ is $F_\sigma$.
The parameters whose fibres lie in $U$ form the open set
\begin{equation}
 \T^1\setminus h(\T^2\setminus U).
\end{equation}
The sets of continua meeting an open set or lying in an open set generate the topology of the Hausdorff metric.
This proves that the fibre map is Borel. Equation~\eqref{eq:lc-hyperspace-class} then shows that $L$ is Borel.

A Jordan curve is locally connected, so $J\subset L$. The first paragraph gives $L\subset J_c$. If
$\theta\in J_c\setminus E$, then $h^{-1}(\theta)=C_\theta$ is a Jordan curve. This proves
\eqref{eq:regular-parameters-modulo-countable}. Both $J$ and $J_c$ differ from $L$ by countable sets.
They are therefore Borel. Finally,
\begin{equation}
 f(h^{-1}(\theta))=h^{-1}(\theta+\alpha),
 \qquad f(C_\theta)=C_{\theta+\alpha},
\end{equation}
so all three parameter sets are invariant under $R_\alpha$.
If either $L$ or $J_c$ is nonmeagre, then $J$ is nonmeagre because these sets differ only by countable sets.
Lemma~\ref{lem:category-zero-one} now shows that $J$ is residual.
\end{proof}

\begin{proof}[Proof of Corollary~\ref{cor:locally-connected}]
Lemma~\ref{lem:lc-parameters} makes the set of Jordan fibres residual. Theorem~\ref{thm:generic-jordan-fibres}
then gives a unique minimal set.
\end{proof}

In Theorems~\ref{thm:thin-counterexample} and~\ref{thm:dehn-counterexample}, the parameters whose fibres are
Jordan curves have full Lebesgue measure. They form a meagre set.
Otherwise Corollary~\ref{cor:locally-connected} would give a unique minimal set. Their complement is therefore residual and has measure zero.

\section{Uniqueness of the minimal set}\label{sec:uniqueness}

Throughout this section, $f$ and $h$ satisfy one of the two dynamical hypotheses of
Theorem~\ref{thm:generic-jordan-fibres}, and every fibre is a thin essential annular continuum.

We adapt the oriented-gap argument of \cite[Proposition~4.2]{BeguinCrovisierJagerLeRoux2009}. The proof has two steps.
First, an invariant open set meeting one fibre in finitely many components must be fully essential.
Second, two distinct minimal sets would give two disjoint invariant open sets with this property.
Two fully essential domains in the torus cannot be disjoint.

\subsection{Invariant open sets}

Recall that a domain $U\subset\T^2$ is fully essential if its complement is inessential.
For a domain, this means that $\T^2\setminus U$ is contained in a topological disc. Every closed curve in the torus is then homotopic to a
curve in $U$. Given two fully essential domains $U,V$, choose a loop in $U$ homotopic to a horizontal circle
and a loop in $V$ homotopic to a vertical circle. Their algebraic intersection number has absolute value $1$,
so $U\cap V\ne\varnothing$.

\begin{lemma}\label{lem:invariant-open-finite-fibre}
Let $f$ and $h$ satisfy one of the two dynamical hypotheses of Theorem~\ref{thm:generic-jordan-fibres}, and suppose that
every fibre has empty interior. Let $V$ be a nonempty invariant open set. If $V\cap h^{-1}(\theta)$ has finitely many
connected components for some $\theta$, then $V$ is connected and fully essential.
\end{lemma}

\begin{proof}
Let $W$ be a component of $V$. Every fibre has empty interior, so $h$ is not constant on $W$.
Since $W$ is connected, $h(W)$ contains an open arc $I$. The map $f$ permutes the components of $V$.
Suppose that the iterates $f^n(W)$ are all distinct. Irrationality gives $\theta-n\alpha\in I$
for infinitely many $n\geq0$. Thus infinitely many of these iterates meet $h^{-1}(\theta)$.
Each connected subset of $V\cap h^{-1}(\theta)$ lies in one component of $V$.
The distinct iterates therefore give infinitely many components of $V\cap h^{-1}(\theta)$.
This contradicts the hypothesis. Hence $f^p(W)=W$ for some
$p\geq1$.

Choose $x\in W$ and a path $\eta$ in $W$ from $x$ to $f^p(x)$. A real lift of $h\circ\eta$ has endpoint
difference $p\alpha+\ell$ for some $\ell\in\Z$. This difference is nonzero because $\alpha$ is irrational. Since
$h\circ f^{jp}\circ\eta=R_{jp\alpha}\circ h\circ\eta$, each iterated path contributes the same difference.
Thus the concatenated path
\begin{equation}
 \eta_n=\eta*(f^p\circ\eta)*\cdots*(f^{(n-1)p}\circ\eta)
\end{equation}
lies in $W$, and a real lift of $h\circ\eta_n$ has endpoint difference $n(p\alpha+\ell)$.

Fix a real representative $t$ of $\theta$. For large $n$, this lift crosses more distinct levels $t+m$, $m\in\Z$,
than there are connected components of $V\cap h^{-1}(\theta)$. Two crossing points at distinct levels $t+m$ and
$t+m'$ therefore lie in the same component $C$ of $V\cap h^{-1}(\theta)$. Since $C$ is connected and meets $W$,
it lies in $W$. The segment of $\eta_n$ between these crossings projects under $h$ to a loop with winding
$m'-m\ne0$.

Choose a small open arc $J$ containing $\theta$. The component of $W\cap h^{-1}(J)$ containing $C$ is open and
connected, hence path connected. Close the segment by a path in this component. The image of the closing path under $h$ stays in $J$ and has both endpoints at $\theta$.
It therefore has zero winding. The resulting loop in $W$ has nonzero winding under $h$.
Perturb it within $W$ to have finitely many transverse self-intersections and split it into simple loops.
Winding is additive, so at least one of these loops, denoted by $\gamma$, still has nonzero winding under $h$.
It is essential in the torus and has a primitive winding vector $v\in\Z^2$.

We now show that $W$ is fully essential. We use an elementary fact: if a torus domain contains two loops with
nonzero algebraic intersection, then it is fully essential. Indeed, a compact subsurface in the domain containing the loops has genus one.
It contains a torus with one disc removed.

In the Dehn-twist case, write $v=(a,b)$. Since $h\simeq\pr_2$, the winding of $h\circ\gamma$ is $b$, so $b\ne0$.
Both $\gamma$ and $f^p(\gamma)$ lie in $W$. Their winding vectors are $(a,b)$ and $(a+kp b,b)$, whose determinant
is $-kp b^2\ne0$. Thus $W$ is fully essential.

In the identity case, suppose that $W$ is not fully essential. Choose $w\in\Z^2$ with $\det(v,w)=1$ and pass to
the annular cover $\R^2/\Z v$. Let $\widehat W$ be a lift component of $W$ containing a lift
$\widehat\gamma$ of $\gamma$. The sets $\widehat W+n w$, $n\in\Z$, are pairwise disjoint.
Otherwise, $\widehat W$ contains a point and a translate of that point by a nonzero multiple of $w$.
Join them by a path in $\widehat W$. Its projection is a loop in $W$ whose winding vector is not parallel to $v$.
Together with $\gamma$, this loop would force $W$ to be fully essential.

The curves $\widehat\gamma+n w$ are therefore disjoint and ordered by $n$. Since $\widehat W$ contains
$\widehat\gamma$ and avoids $\widehat\gamma-w$ and $\widehat\gamma+w$, it lies between these two neighbouring
curves. Its transverse coordinate is bounded. Let $F$ be a lift of $f$, with rotation vector $\rho$.
The map induced by $F^p$ on the annular cover sends $\widehat W$ onto $\widehat W+jw$ for some $j\in\Z$.
Since $F$ commutes with deck translations, its $pn$-th iterate sends $\widehat W$ onto $\widehat W+n j w$.
Thus, for a point $z\in\R^2$ projecting into $\widehat W$,
\begin{equation}
 \det\bigl(v,F^{pn}(z)-z\bigr)=nj+O(1).
\end{equation}
Dividing by $pn$ and using \eqref{eq:uniform-rotation} gives $\det(v,\rho)=j/p$, contrary to total irrationality.
Hence $W$ is fully essential in this case as well.

Every component of $V$ is therefore fully essential. Any two such domains meet, by the intersection argument
above. Since distinct components are disjoint, $V$ is connected.
\end{proof}

\subsection{Prime ends and Jordan curves}

The next lemma concerns the geometry of the fibres. It does not use $f$.
We fix a compact set meeting every fibre. It need not be invariant.
At a suitable parameter, the fibre is a Jordan curve. The lemma controls every impression of every nearby fibre on both sides.
The nearby fibres may be singular. This control allows gaps between disjoint compact sets to persist on nearby prime-end circles.

Work over a bounded parameter interval in the cyclic cover $q:\mathbb A\to\T^2$. Recall that
$A_t=\widehat h^{-1}(t)$. The sets $U^-(A_t)$ and $U^+(A_t)$ are the components of $\mathbb A\setminus A_t$
containing the lower and upper ends, respectively. Their boundaries are $A_t^\pm=\partial U^\pm(A_t)$.
A compact set $K\subset\T^2$ lies in the torus. Its preimage $q^{-1}(K)$ lies in the annular cover.
We use $q^{-1}(K)$ when taking intersections with $A_t$, its side boundaries or prime-end impressions.
Write $(x,y)\in\R\times\T^1$ for annular coordinates, with $x$ increasing towards the upper end.
Cap the lower end of $U^-(A_t)$ and use the planar coordinate $z=\exp(2\pi(x+\mathrm i y))$.
The resulting domains $D_t^-$ are bounded and simply connected, and contain a common disc about $0$.
For $U^+(A_t)$, cap the upper end and use $z=\exp(-2\pi(x+\mathrm i y))$.
Denote the resulting domains by $D_t^+$.
For $\sigma\in\{-,+\}$, let
$\varphi_t^\sigma:\mathbb D\to D_t^\sigma$ be the Riemann maps normalised by $\varphi_t^\sigma(0)=0$ and
$(\varphi_t^\sigma)'(0)>0$. Here the complex derivative is a positive real number.
All boundary comparisons take place in a fixed compact annular region.

Apply Lemma~\ref{lem:prime-end-facts} to each $D_t^\sigma$. Denote the impression at
$\zeta\in\partial\mathbb D$ by $I_t^\sigma(\zeta)$, regarded as a subset of $A_t^\sigma$. It is the full
cluster set of $\varphi_t^\sigma$ at $\zeta$. These impressions cover $A_t^\sigma$. For each fixed $t$, the set
$\{(\zeta,x)\in\partial\mathbb D\times A_t^\sigma:x\in I_t^\sigma(\zeta)\}$ is compact.

Whenever $A_t$ is a Jordan curve, write $F_t^\sigma$ for the boundary extension of $\varphi_t^\sigma$. We regard its values as
points of $A_t$, using the fixed planar coordinate for uniform norms. For a compact set $K\subset\T^2$ with
$h(K)=\T^1$, put
\begin{equation}\label{eq:prime-end-colour-sets}
 P_K^\sigma(t)=\{\zeta\in\partial\mathbb D:I_t^\sigma(\zeta)\cap q^{-1}(K)\ne\varnothing\}.
\end{equation}
Each $P_K^\sigma(t)$ is compact. Indeed, the pairs $(\zeta,x)$ with
$x\in I_t^\sigma(\zeta)\cap q^{-1}(K)$ form a compact set. Its projection onto the first coordinate is $P_K^\sigma(t)$.
These sets are also nonempty. To see this, take $s_n\nearrow t$ and choose
$x_n\in q^{-1}(K)\cap A_{s_n}$. By \eqref{eq:left-Hausdorff}, a subsequence converges to a point of $A_t^-$.
The limit also belongs to $q^{-1}(K)$, which is closed. Taking $s_n\searrow t$ and using
\eqref{eq:right-Hausdorff} gives a point of $q^{-1}(K)\cap A_t^+$. Since the impressions cover each side boundary,
$P_K^\sigma(t)$ is nonempty for both signs.

\begin{lemma}\label{lem:prime-end-good-fibre}
Let $h:\T^2\to\T^1$ be continuous and homotopic to a coordinate projection.
Suppose that every fibre is a thin essential annular continuum and that the fibres are Jordan curves on a residual set of parameters.
Let $K\subset\T^2$ be compact with $h(K)=\T^1$.
Then there is a residual set $G_K\subset\T^1$ such that, for every real lift $t_0$ of a parameter in $G_K$,
the fibre $A_{t_0}$ is a Jordan curve and the following conclusions hold.
In the limits below, $t$ varies through all nearby parameters, including those for which $A_t$ is not a Jordan curve.
\begin{enumerate}
\item The sets $q^{-1}(K)\cap A_t$ converge to $q^{-1}(K)\cap A_{t_0}$ in the Hausdorff metric as $t\to t_0$.
For each $\sigma\in\{-,+\}$, the same is true with $A_t$ replaced by $A_t^\sigma$.
\item For each $\sigma\in\{-,+\}$, the sets $P_K^\sigma(t)$ converge to $P_K^\sigma(t_0)$ in the Hausdorff metric as $t\to t_0$.
\item For every $\varepsilon>0$ there is $\delta>0$ such that, for every $t$ with $|t-t_0|<\delta$ and both signs $\sigma\in\{-,+\}$,
\begin{equation}\label{eq:prime-end-graph-estimate}
 \sup_{\zeta\in\partial\mathbb D}\ \sup_{x\in I_t^\sigma(\zeta)}
 d\bigl(x,F_{t_0}^\sigma(\zeta)\bigr)<\varepsilon.
\end{equation}
\end{enumerate}
\end{lemma}

\begin{proof}
Cover the base by countably many open intervals on which the parameter has a real lift. On each interval it is
enough to exclude a meagre set; the union of the resulting exceptional sets is meagre in the base.
Fix one such interval and choose a dense $G_\delta$ subset $G$ whose fibres are Jordan curves.
We first choose the parameters $t_0$ at which the conclusions will hold.

For $t\in G$, both side boundaries equal $A_t$. The one-sided limits in
\eqref{eq:left-Hausdorff} and \eqref{eq:right-Hausdorff} therefore imply that $A_s\to A_t$ in the Hausdorff metric as $s\to t$.
In particular, the boundaries of the capped domains converge when $s\to t$ within $G$.
Lemma~\ref{lem:riemann-map-convergence}\eqref{item:jordan-domain-limit} gives convergence of
$\varphi_s^\sigma$ to $\varphi_t^\sigma$ uniformly on each compact subset of $\mathbb D$.
Thus, for each $0<r<1$, the map $t\mapsto(\zeta\mapsto\varphi_t^\sigma(r\zeta))$ is continuous on $G$ in the uniform norm.
For each fixed $t\in G$, these maps converge uniformly on $\partial\mathbb D$ to $F_t^\sigma$ as $r\nearrow1$.
It follows that $t\mapsto F_t^\sigma$ is a Baire-one map from $G$ to
$C(\partial\mathbb D,\mathbb C)$, equipped with the uniform norm.

The map $t\mapsto q^{-1}(K)\cap A_t$ is upper semicontinuous on the full interval.
Indeed, these nonempty compact sets lie in a fixed compact annular region, and any limit point of their elements belongs to the limiting fibre and to $q^{-1}(K)$.
Apply Lemma~\ref{lem:baire-continuity} to this map on the full interval and to both boundary maps on $G$.
Since $G$ is a dense $G_\delta$ set, there is a residual set of parameters $t_0\in G$ at which the intersection map is continuous and both boundary maps are continuous relative to $G$.
Fix such a parameter $t_0$.
The first convergence in part~(1) is exactly the continuity of the intersection map at $t_0$.

We first prove \eqref{eq:prime-end-graph-estimate}.
Fix $\varepsilon>0$. Choose $\eta>0$ so that planar distance at most $\eta$ between points of the fixed compact annular region gives annular distance at most $\varepsilon/2$, in either chart.
Continuity of both boundary maps at $t_0$ gives $\delta>0$ such that
$\|F_s^\sigma-F_{t_0}^\sigma\|_\infty<\eta$ whenever $s\in G$, $|s-t_0|<2\delta$ and $\sigma\in\{-,+\}$.
Choose $\delta$ small enough that this neighbourhood lies in the parameter interval.
The maximum principle gives
$|\varphi_s^\sigma(z)-\varphi_{t_0}^\sigma(z)|<\eta$ for every $z\in\mathbb D$ and the same parameters $s$.

Now fix any $t$ with $|t-t_0|<\delta$. No assumption is made on the fibre $A_t$.
Since $G$ is dense, choose $s_n^-\in G$ increasing to $t$ and $s_n^+\in G$ decreasing to $t$, with $|s_n^\sigma-t_0|<2\delta$.
The descriptions of the complementary domains in \eqref{eq:fibre-complement-sides} give
\begin{equation}\label{eq:increasing-capped-domains}
 D_t^-=\bigcup_n D_{s_n^-}^-,\qquad
 D_t^+=\bigcup_n D_{s_n^+}^+.
\end{equation}
Both unions are increasing. Lemma~\ref{lem:riemann-map-convergence}\eqref{item:increasing-domains} therefore gives
$\varphi_{s_n^\sigma}^\sigma\to\varphi_t^\sigma$ uniformly on each compact subset of $\mathbb D$.
Passing to the limit in the preceding bound yields
$|\varphi_t^\sigma(z)-\varphi_{t_0}^\sigma(z)|\leq\eta$ for every $z\in\mathbb D$ and both signs.

Fix $\zeta\in\partial\mathbb D$ and $x\in I_t^\sigma(\zeta)$.
By the definition of an impression, there are points $z_n\in\mathbb D$ tending to $\zeta$ such that
$\varphi_t^\sigma(z_n)\to x$ in the corresponding planar chart.
Since $A_{t_0}$ is a Jordan curve, the continuous boundary extension gives
$\varphi_{t_0}^\sigma(z_n)\to F_{t_0}^\sigma(\zeta)$.
The planar distance between $x$ and $F_{t_0}^\sigma(\zeta)$ is therefore at most $\eta$.
Their annular distance is at most $\varepsilon/2$. This bound is uniform in $t$, $\sigma$, $\zeta$ and $x$ in the specified ranges, so it proves \eqref{eq:prime-end-graph-estimate}.

We next prove convergence of the intersections with the side boundaries:
\begin{equation}\label{eq:prime-end-side-colours}
 q^{-1}(K)\cap A_t^\sigma\longrightarrow q^{-1}(K)\cap A_{t_0}
 \quad\text{in the Hausdorff metric as }t\to t_0.
\end{equation}
Every limit point belongs to $q^{-1}(K)\cap A_{t_0}$, since
$q^{-1}(K)\cap A_t^\sigma\subset q^{-1}(K)\cap A_t$.
For the converse, fix $\varepsilon>0$.
Continuity of $s\mapsto q^{-1}(K)\cap A_s$ at $t_0$ gives a neighbourhood in which every point of
$q^{-1}(K)\cap A_{t_0}$ is within $\varepsilon/2$ of $q^{-1}(K)\cap A_s$.
Fix $t$ in a smaller neighbourhood and $y\in q^{-1}(K)\cap A_{t_0}$.
Take $s_n\nearrow t$ within the first neighbourhood and choose $x_n\in q^{-1}(K)\cap A_{s_n}$ with $d(x_n,y)<\varepsilon/2$.
A subsequence converges to a point of $q^{-1}(K)\cap A_t^-$ by \eqref{eq:left-Hausdorff}.
This point is within $\varepsilon$ of $y$.
Taking $s_n\searrow t$ and using \eqref{eq:right-Hausdorff} gives the same conclusion for $A_t^+$.
The neighbourhood is independent of $y$, which proves \eqref{eq:prime-end-side-colours}.

We now prove convergence of the prime-end sets. Fix $\sigma\in\{-,+\}$.
Suppose that $t_n\to t_0$ and $\zeta_n\in P_K^\sigma(t_n)$ converge to $\zeta$.
Choose $x_n\in I_{t_n}^\sigma(\zeta_n)\cap q^{-1}(K)$.
By \eqref{eq:prime-end-graph-estimate} and continuity of $F_{t_0}^\sigma$, we have
$x_n\to F_{t_0}^\sigma(\zeta)$.
The limit belongs to $q^{-1}(K)$, so $\zeta\in P_K^\sigma(t_0)$.

Conversely, let $\zeta\in P_K^\sigma(t_0)$.
By \eqref{eq:prime-end-side-colours}, choose $x_t\in q^{-1}(K)\cap A_t^\sigma$ such that
$d(x_t,F_{t_0}^\sigma(\zeta))\to0$, uniformly over $\zeta\in P_K^\sigma(t_0)$.
Each $x_t$ belongs to an impression $I_t^\sigma(\zeta_t)$, so $\zeta_t\in P_K^\sigma(t)$.
Equation~\eqref{eq:prime-end-graph-estimate} gives
$d(F_{t_0}^\sigma(\zeta_t),F_{t_0}^\sigma(\zeta))\to0$ uniformly over these $\zeta$.
The inverse of $F_{t_0}^\sigma$ is uniformly continuous. Hence $\zeta_t\to\zeta$ uniformly as well.
Together with the preceding limit-point argument, this proves the claimed Hausdorff convergence.
\end{proof}

\subsection{Oriented gaps}

For two disjoint compact subsets $E_0,E_1$ of an oriented circle, a gap is a component of the complement of
$E_0\cup E_1$. It has type $ij$, with $ij\in\{01,10\}$, if its
initial endpoint belongs to $E_i$ and its terminal endpoint belongs to $E_j$. We call these \textbf{mixed gaps}.

\begin{proof}[Proof of Theorem~\ref{thm:generic-jordan-fibres}]
Assume for a contradiction that $M_0$ and $M_1$ are distinct minimal sets.

\smallskip
\noindent\emph{Step 1: choose a common Jordan fibre.}
The sets $M_0$ and $M_1$ are disjoint. For each $i\in\{0,1\}$, the image $h(M_i)$ is a nonempty compact
$R_\alpha$-invariant subset of the circle. Minimality of $R_\alpha$ therefore gives $h(M_i)=\T^1$.
Apply Lemma~\ref{lem:prime-end-good-fibre} separately to $M_0$ and $M_1$, and choose a real lift $t_0$ of a
parameter in the intersection of the two residual sets. Then $A_{t_0}$ is a Jordan curve, and all three conclusions of the lemma hold for both
minimal sets. Write $P_i^\sigma(t)=P_{M_i}^\sigma(t)$ for $i\in\{0,1\}$ and $\sigma\in\{-,+\}$.
Since impressions on $A_{t_0}$ are single points, the disjointness of $M_0$ and $M_1$ implies that
$P_0^\sigma(t_0)$ and $P_1^\sigma(t_0)$ are disjoint. They are nonempty compact sets. Their Hausdorff continuity
shows that they remain disjoint for all $t$ sufficiently close to $t_0$.

\smallskip
\noindent\emph{Step 2: define two invariant gap sets.}
Orient both prime-end circles in the common annular direction. Their orientations agree when the fibre is a
Jordan curve, although for one capped domain this reverses the analytic boundary orientation. Since
$\widehat h\circ\widehat f=\widehat h+\alpha$, the lift $\widehat f$ preserves both ends. It therefore induces
homeomorphisms of the prime-end circles which carry impressions to impressions. In both homotopy classes,
$f_*$ is the identity on $\Lambda_h$, so these homeomorphisms preserve the common orientation.

For any parameter $t$, a complementary gap of $P_0^\sigma(t)\cup P_1^\sigma(t)$ has type $ij$ if its initial
endpoint lies in $P_i^\sigma(t)\setminus P_j^\sigma(t)$ and its terminal endpoint lies in
$P_j^\sigma(t)\setminus P_i^\sigma(t)$. This definition still makes sense when the two compact sets overlap.
For $ij\in\{01,10\}$, define $U_{ij}\subset\T^2$ as follows. If $x\in A_t$, then $q(x)\in U_{ij}$ when every
prime end whose impression contains $x$ lies in a gap of type $ij$. This is required on every side
$\sigma\in\{-,+\}$ for which $x\in A_t^\sigma$. The condition is imposed on every such prime end because a point
of a singular fibre may belong to several impressions.

This definition is unchanged under deck translations and hence does not depend on the lift $x$ of $q(x)$.
Every point belongs to an impression on at least one side, so $U_{01}\cap U_{10}=\varnothing$. The dynamics
preserves the sides, their orientations, the impressions and both minimal sets. Thus each $U_{ij}$ is invariant.
Set $V_{ij}=\operatorname{int}U_{ij}$.

\smallskip
\noindent\emph{Step 3: show that both interiors are nonempty and have finite fibre sections.}
Let $\mathcal G_{ij}$ be the gaps of type $ij$ in
$A_{t_0}\setminus q^{-1}(M_0\cup M_1)$. On the Jordan curve $A_{t_0}$, the two side definitions give the same
cyclic order. In either prime-end parametrisation, the two compact sets are positively separated. Hence the
mixed gaps have a common positive lower bound on their angular lengths. The two families $\mathcal G_{01}$ and
$\mathcal G_{10}$ are therefore finite. Both are nonempty because the two compact sets are nonempty and disjoint.

We prove that
\begin{equation}\label{eq:whole-mixed-gap-sections}
 q^{-1}(V_{ij})\cap A_{t_0}=q^{-1}(U_{ij})\cap A_{t_0}
   =\bigcup_{J\in\mathcal G_{ij}}J.
\end{equation}
The second equality follows directly from the definition because every impression on the Jordan fibre $A_{t_0}$
is a single point. It remains to prove that each point of a mixed gap is interior to the corresponding set.

Fix $x\in J\in\mathcal G_{ij}$. On each prime-end circle, choose a small compact arc $L^\sigma$ around the
unique prime end representing $x$, with $L^\sigma$ contained in the corresponding gap. Choose disjoint closed
arc neighbourhoods of the two endpoints of that gap. The first avoids $P_j^\sigma(t_0)$, the second avoids
$P_i^\sigma(t_0)$, and the closed arc between them contains $L^\sigma$ and avoids
$P_0^\sigma(t_0)\cup P_1^\sigma(t_0)$. The Hausdorff convergence in Lemma~\ref{lem:prime-end-good-fibre}
preserves these exclusions for nearby $t$. It also ensures that the two endpoint neighbourhoods still meet
$P_i^\sigma(t)$ and $P_j^\sigma(t)$, respectively. Thus $L^\sigma$ remains inside a gap of type $ij$.

For each sign, the compact set
$F_{t_0}^\sigma(\partial\mathbb D\setminus\operatorname{int}L^\sigma)$ does not contain $x$ and hence has
positive distance from $x$. The estimate \eqref{eq:prime-end-graph-estimate} now shows that every prime end whose
impression contains a point $x'$ sufficiently close to $x$ lies in $L^\sigma$. Here
$t=\widehat h(x')$ is close to $t_0$. The same argument applies on every side containing $x'$. Hence every such
$x'$ belongs to $q^{-1}(U_{ij})$. Thus $x$ is an interior point of $q^{-1}(U_{ij})$, and the local
homeomorphism $q$ gives $q(x)\in V_{ij}$. This proves \eqref{eq:whole-mixed-gap-sections}.

It follows that $V_{01}$ and $V_{10}$ are disjoint, nonempty, invariant open sets.
The restriction of $q$ to $A_{t_0}$ is a homeomorphism onto $h^{-1}(t_0\bmod1)$.
Equation~\eqref{eq:whole-mixed-gap-sections} shows that each $V_{ij}$ meets this torus fibre in finitely many
connected components.

\smallskip
\noindent\emph{Step 4: obtain the contradiction.}
Lemma~\ref{lem:invariant-open-finite-fibre} shows that both $V_{01}$ and $V_{10}$ are connected and fully essential.
Two fully essential domains in the torus must intersect, whereas $V_{01}$ and $V_{10}$ are disjoint. This
contradiction proves that $f$ has a unique minimal set.
\end{proof}

\subsection{Factors of the Denjoy model}

The examples in Theorem~\ref{thm:thin-counterexample} are obtained as limits of conjugates of Avila's model $g$.
A continuous surjection $q:\T^2\to\T^2$ satisfying $q\circ g=f\circ q$ would instead realise $f$ as a torus factor
of $g$. The following corollary rules this out whenever $f$ is a totally irrational pseudo-rotation admitting an
irrational circle factor with thin fibres.

\begin{corollary}\label{cor:denjoy-no-thin-factor}
Let $g:\T^2\to\T^2$ be the map constructed in Proposition~\ref{prop:denjoy-model}, and let $f:\T^2\to\T^2$ be a totally irrational
pseudo-rotation admitting an irrational circle factor $h:\T^2\to\T^1$, with $h\simeq\pr_1$ and every fibre thin.
Then there is no continuous surjection $q:\T^2\to\T^2$ satisfying
\begin{equation}\label{eq:denjoy-compatible-factor}
 q\circ g=f\circ q.
\end{equation}
No monotonicity assumption on $q$ is required.
\end{corollary}

\begin{proof}
Recall that $g(x,s)=(f_0(x),s+u(x))$, $h_0(x,s)=\psi(x)$, and the minimal sets of $g$ are $K\times\{s\}$.
We first show that every continuous $g$-invariant map $A:\T^2\to\T^1$ is constant. On $K\times\{s\}$, minimality
gives $A(x,s)=a(s)$. For $x\in\operatorname{int}I^0$, all backward iterates have second coordinate $s$, and all
positive iterates have second coordinate $s+u(x)$: the other intervals in this wandering orbit carry integer values
of $u$. These iterates approach $K\times\T^1$, so invariance and continuity give $a(s)=a(s+u(x))$.
By continuity this also holds at the endpoints of $I^0$. The values of $u$ on $I^0$ cover an interval of length one,
so $a$ is constant. Every orbit outside $K\times\T^1$ accumulates on that set, so $A$ is constant on the whole torus.

Suppose that $q$ exists, write $h\circ f=R_\gamma\circ h$ with $\gamma$ irrational, and put $H=h\circ q$.
Fix $s_0\in\T^1$. The endpoints of each complementary interval of $K$ are forward asymptotic under $f_0$.
Since rotations preserve distance, $H(\cdot,s_0)$ takes the same value at each such pair. It therefore descends
through $\psi|_K$ to a continuous map $\chi:\T^1\to\T^1$ satisfying $\chi\circ R_\alpha=R_\gamma\circ\chi$.
If $m$ is its degree, write a lift as $\widetilde\chi(t)=mt+v(t)$ with $v$ periodic. The factor identity makes
$v(t+\alpha)-v(t)$ constant. Boundedness forces this constant to vanish; irrationality then makes $v$ constant.
Hence $\gamma=m\alpha\pmod1$ and $m\neq0$. The map $H-mh_0$ is now $g$-invariant, so the preceding observation gives
\begin{equation}\label{eq:denjoy-factor-form}
 h\circ q=mh_0+c\qquad\text{for some }c\in\T^1.
\end{equation}

Let $D\subset\T^1$ be the union of the two blown-up rotation orbits and put $B=\{m\theta+c:\theta\in D\}$.
For $t\notin B$, surjectivity of $q$ and \eqref{eq:denjoy-factor-form} give
\begin{equation}\label{eq:denjoy-target-fibre}
 h^{-1}(t)=\bigcup_{m\theta+c=t}q\bigl(h_0^{-1}(\theta)\bigr).
\end{equation}
Each source fibre in this union is a circle. Its continuous image is locally connected, and a finite union of
compact locally connected sets is locally connected. Thus every target fibre outside the countable set $B$
is locally connected. Lemma~\ref{prop:canonical-fibres} and Corollary~\ref{cor:locally-connected} give a unique
minimal set $M$ for $f$.

Every set $q(K\times\{s\})$ is minimal, so equals $M$. Hence $q(K\times\T^1)=M$.
For $\theta\notin D$, the source fibre $h_0^{-1}(\theta)$ lies in $K\times\T^1$. Thus
\eqref{eq:denjoy-target-fibre} places every target fibre outside $B$ in $M$. Their union is dense, since the
omitted fibres are countably many closed sets with empty interior. Therefore $M=\T^2$.

Fix $s\in\T^1$ and $t\notin B$. The restriction of $q$ to $K\times\{s\}$ is onto $\T^2$.
There are exactly $|m|$ solutions of $m\theta+c=t$, and each has a unique preimage $x_\theta$ under $\psi$.
By \eqref{eq:denjoy-factor-form}, every point of $h^{-1}(t)$ is therefore one of the points $q(x_\theta,s)$.
This makes $h^{-1}(t)$ finite, contradicting Lemma~\ref{prop:canonical-fibres}.
\end{proof}

\bibliographystyle{alpha}
\bibliography{minimal}

@misc{AvilaPrivateCommunication,
  author = {Avila, Artur},
  note   = {Private communication}
}

@article{AddasZanataTalGarcia2014,
  author  = {Addas-Zanata, Salvador and Tal, F{\'a}bio A. and Garcia, Br{\'a}ulio A.},
  title   = {Dynamics of homeomorphisms of the torus homotopic to {Dehn} twists},
  journal = {Ergodic Theory and Dynamical Systems},
  volume  = {34},
  number  = {2},
  pages   = {409--422},
  year    = {2014},
  doi     = {10.1017/etds.2012.156}
}

@article{HammerlindlPotrie2014,
  author  = {Hammerlindl, Andy and Potrie, Rafael},
  title   = {Pointwise partial hyperbolicity in three-dimensional nilmanifolds},
  journal = {Journal of the London Mathematical Society},
  series  = {2},
  volume  = {89},
  number  = {3},
  pages   = {853--875},
  year    = {2014},
  doi     = {10.1112/jlms/jdu013}
}

@article{BeguinCrovisierJagerLeRoux2009,
  author  = {B{\'e}guin, Fran{\c c}ois and Crovisier, Sylvain and
             J{\"a}ger, Tobias and Le Roux, Fr{\'e}d{\'e}ric},
  title   = {Denjoy constructions for fibered homeomorphisms of the torus},
  journal = {Transactions of the American Mathematical Society},
  volume  = {361},
  number  = {11},
  pages   = {5851--5883},
  year    = {2009},
  doi     = {10.1090/S0002-9947-09-04914-9}
}

@incollection{BeguinCrovisierJager2017,
  author    = {B{\'e}guin, Fran{\c c}ois and Crovisier, Sylvain and
               J{\"a}ger, Tobias},
  title     = {A dynamical decomposition of the torus into pseudo-circles},
  booktitle = {Modern Theory of Dynamical Systems: A Tribute to
               Dmitry Victorovich Anosov},
  series    = {Contemporary Mathematics},
  volume    = {692},
  pages     = {39--50},
  publisher = {American Mathematical Society},
  address   = {Providence, RI},
  year      = {2017},
  doi       = {10.1090/conm/692/13915}
}

@article{JagerKeller2006,
  author  = {J{\"a}ger, Tobias H. and Keller, Gerhard},
  title   = {The {Denjoy} type of argument for quasiperiodically forced
             circle diffeomorphisms},
  journal = {Ergodic Theory and Dynamical Systems},
  volume  = {26},
  number  = {2},
  pages   = {447--465},
  year    = {2006},
  doi     = {10.1017/S0143385705000477}
}

@article{JagerPasseggi2015,
  author  = {J{\"a}ger, Tobias and Passeggi, Alejandro},
  title   = {On torus homeomorphisms semiconjugate to irrational rotations},
  journal = {Ergodic Theory and Dynamical Systems},
  volume  = {35},
  number  = {7},
  pages   = {2114--2137},
  year    = {2015},
  doi     = {10.1017/etds.2014.23}
}

@article{JagerKwakkelPasseggi2013,
  author  = {J{\"a}ger, Tobias and Kwakkel, Ferry and Passeggi, Alejandro},
  title   = {A classification of minimal sets of torus homeomorphisms},
  journal = {Mathematische Zeitschrift},
  volume  = {274},
  number  = {1--2},
  pages   = {405--426},
  year    = {2013},
  doi     = {10.1007/s00209-012-1076-y}
}

@article{KoropeckiPasseggiSambarino2021,
  author  = {Koropecki, Andr{\'e}s and Passeggi, Alejandro and Sambarino, Mart{\'i}n},
  title   = {The {Franks--Misiurewicz} conjecture for extensions of
             irrational rotations},
  journal = {Annales scientifiques de l'\'{E}cole normale sup{\'e}rieure},
  series  = {4},
  volume  = {54},
  number  = {4},
  pages   = {1035--1049},
  year    = {2021},
  doi     = {10.24033/asens.2476}
}

@article{Potrie2012,
  author  = {Potrie, Rafael},
  title   = {Recurrence of non-resonant homeomorphisms on the torus},
  journal = {Proceedings of the American Mathematical Society},
  volume  = {140},
  number  = {11},
  pages   = {3973--3981},
  year    = {2012},
  doi     = {10.1090/S0002-9939-2012-11249-3}
}

@article{Kwakkel2011,
  author  = {Kwakkel, Ferry H.},
  title   = {Minimal sets of non-resonant torus homeomorphisms},
  journal = {Fundamenta Mathematicae},
  volume  = {211},
  number  = {1},
  pages   = {41--76},
  year    = {2011},
  doi     = {10.4064/fm211-1-3},
  note    = {Corrected version: arXiv:1002.0364v3; erratum in
             Fundamenta Mathematicae 213 (2011), 291}
}

@article{HrushovskiLoeserPoonen2014,
  author = {Hrushovski, Ehud and Loeser, Fran{\c c}ois and Poonen, Bjorn},
  title = {Berkovich spaces embed in {Euclidean} spaces},
  journal = {L'Enseignement Math{\'e}matique},
  volume = {60},
  number = {3--4},
  year = {2014},
  pages = {273--292},
  doi = {10.4171/LEM/60-3/4-4},
  eprint = {1210.6485},
  archivePrefix = {arXiv},
  primaryClass = {math.AG}
}

@article{Rempe2008,
  author = {Rempe, Lasse},
  title = {On prime ends and local connectivity},
  journal = {Bulletin of the London Mathematical Society},
  volume = {40},
  number = {5},
  pages = {817--826},
  year = {2008},
  doi = {10.1112/blms/bdn061},
  eprint = {math/0309022},
  archivePrefix = {arXiv},
  primaryClass = {math.GN}
}

@book{Pommerenke1992,
  author = {Pommerenke, Christian},
  title = {Boundary Behaviour of Conformal Maps},
  series = {Grundlehren der mathematischen Wissenschaften},
  volume = {299},
  publisher = {Springer-Verlag},
  address = {Berlin},
  year = {1992},
  doi = {10.1007/978-3-662-02770-7}
}

@article{LeRoux2014,
  author = {Le Roux, Fr{\'e}d{\'e}ric},
  title = {On closed subgroups of the group of homeomorphisms of a manifold},
  journal = {Journal de l'{\'E}cole polytechnique --- Math{\'e}matiques},
  volume = {1},
  pages = {147--159},
  year = {2014},
  doi = {10.5802/jep.7}
}

@article{Kocsard2021,
  author  = {Kocsard, Alejandro},
  title   = {Periodic point free homeomorphisms and irrational rotation
             factors},
  journal = {Ergodic Theory and Dynamical Systems},
  volume  = {41},
  number  = {10},
  pages   = {2946--2982},
  year    = {2021},
  doi     = {10.1017/etds.2020.88},
  eprint  = {1908.05746},
  archivePrefix = {arXiv},
  primaryClass = {math.DS}
}

@misc{Correa2026,
  author  = {Corr{\^e}a, Heric},
  title   = {Bounded vertical deviations and irrational circle factors in
             {Dehn Twist} classes},
  year    = {2026},
  eprint  = {2609.11627},
  archivePrefix = {arXiv},
  primaryClass = {math.DS},
  note    = {arXiv:2609.11627v1, 10 September 2026}
}

@book{Kechris1995,
  author = {Kechris, Alexander S.},
  title = {Classical Descriptive Set Theory},
  series = {Graduate Texts in Mathematics},
  volume = {156},
  publisher = {Springer-Verlag},
  address = {New York},
  year = {1995},
  doi = {10.1007/978-1-4612-4190-4}
}

@article{FayadKatok2004,
  author  = {Fayad, Bassam and Katok, Anatole},
  title   = {Constructions in elliptic dynamics},
  journal = {Ergodic Theory and Dynamical Systems},
  volume  = {24},
  number  = {5},
  pages   = {1477--1520},
  year    = {2004},
  doi     = {10.1017/S0143385703000798}
}

@article{KoropeckiLeCalvezNassiri2015,
  author  = {Koropecki, Andr{\'e}s and Le Calvez, Patrice and Nassiri, Meysam},
  title   = {Prime ends rotation numbers and periodic points},
  journal = {Duke Mathematical Journal},
  volume  = {164},
  number  = {3},
  pages   = {403--472},
  year    = {2015},
  doi     = {10.1215/00127094-2861386}
}

@incollection{Mather1982,
  author    = {Mather, John N.},
  title     = {Topological proofs of some purely topological consequences
               of {Carath\'{e}odory}'s theory of prime ends},
  booktitle = {Selected Studies: Physics-Astrophysics, Mathematics,
               History of Science},
  editor    = {Rassias, Th. M. and Rassias, G. M.},
  publisher = {North-Holland},
  address   = {Amsterdam},
  pages     = {225--255},
  year      = {1982}
}

\end{document}